\documentclass[11pt]{article}
\usepackage{enumerate}
\usepackage{amssymb,a4wide,latexsym,makeidx,epsfig,fleqn}
\usepackage{amsthm}
\usepackage{amsmath}
\usepackage{enumerate}

\usepackage{enumerate}
\usepackage{mathrsfs}
\usepackage{amssymb,latexsym,makeidx,xypic,epsfig,fleqn}
\usepackage{a4wide}
\usepackage{ctex}
\usepackage{times}
\usepackage{floatrow}
\usepackage{multirow}
\usepackage{amsmath}
\usepackage{enumerate}
\usepackage{graphicx}
\usepackage{subfigure}
\usepackage{arydshln}
\usepackage{caption}
\usepackage[rflt]{floatflt}

\newtheorem{theorem}{Theorem}[section]

\newtheorem{lemma}[theorem]{Lemma}

\graphicspath{{figures/}}

\begin{document}
\textwidth 150mm \textheight 225mm
\title{5-regular oriented graphs with optimum skew energy
\thanks{ Supported by
the National Natural Science Foundation of China (No.11171273)}}
\author{{Lifeng Guo, Ligong Wang, Peng Xiao}\\
{\small Department of Applied Mathematics, School of Science, Northwestern
Polytechnical University,}\\
{\small  Xi'an, Shaanxi 710072, People's Republic of China.}\\
{\small E-mail: lfguomath@163.com, lgwangmath@163.com, xiaopeng@sust.edu.cn}\\
}

\date{}
\maketitle
\begin{center}
\begin{minipage}{120mm}
\vskip 0.3cm
\begin{center}
{\small {\bf Abstract}}
\end{center}
{\small Let $G$ be a simple undirected graph and $G^\sigma$ be the
corresponding oriented graph of $G$ with the orientation $\sigma$.
The skew energy of $G^\sigma$, denoted by $\varepsilon_s(G^\sigma)$,
is defined as the sum of the singular values of the skew adjacency
matrix $S(G^\sigma)$. In 2010, Adiga et al. certified that
$\varepsilon_s(G^\sigma) \leq n\sqrt{\Delta}$, where $\Delta$ is the maximum degree of
$G$ of order $n$. In this paper, we determine all connected 5-regular oriented
graphs of order $n$ with maximum skew-energy.

}
\end{minipage}
\end{center}

\section{Introduction} \label{intro-sec}

Let $G=(V(G),E(G))$ be a simple and undirected graph with vertex set
$V=V(G)=\{v_1,v_2,\ldots ,v_n\}$ and edge set $E(G)=\{e_1,e_2,\ldots
,e_m \}$. For any $v\in V$, we denote by $d(v)=d_G(v)$ and $N(v)=N_G(v)$
respectively the degree and the neighborhood of $v$ in $G$.
Denote by $\Delta=\Delta(G)$ the maximum degree of $G$
and $G[W]$ the subgraph of $G$ induced by the subset $W\subseteq
V(G)$. Let $G^\sigma$ be an oriented graph of $G$ with the
orientation $\sigma$, which assigns to each edge of $G$ a direction
so that the resulting graph $G^\sigma$ becomes an oriented graph.
Denote by $(u, v)$ the arc of $G^\sigma$, with tail $u$ and head $v$.
We call $G$ the underlying graph of $G^\sigma$. Here we take $V(G^\sigma)=V(G)$,
$N_{G^\sigma}(v)=N_G(v)=N(v)$.

  Let $A(G)$ be the $(0,1)$-adjacency matrix of $G$.
The skew-adjacency matrix of $G^\sigma$
is the $n\times n$ matrix $S(G^\sigma)=[s_{ij}]$, $s_{ij}=1$
whenever $(v_i, v_j)\in E(G^\sigma)$, $s_{ij}=-1$ whenever $(v_j,
v_i)\in E(G^\sigma)$, $s_{ij}=0$ otherwise. Denote by
$\varepsilon_s(G^\sigma)$ the skew energy of $G^\sigma$, which is
defined as the sum of the singular values of $S(G^\sigma)$. Since
$S(G^\sigma)$ is a skew-symmetric matrix, the sum of
the singular values of $S(G^\sigma)$ is the sum of the absolute
values of its eigenvalues.

  In \cite{IG1}, Gutman introduced the energy of a simple undirected graph.
Since several results on the energy of the adjacent matrix of a graph
have been obtained \cite{IG1,IG2,LSI}.
Recently more concepts of graph energy are studied, such as Laplacian energy,
signless Laplacian energy, incidence energy, distance energy for an undirected graph, the skew
energy and skew Laplacian energy for an oriented graph, see \cite{IG2} and \cite{LSI}.
The recent achievements about skew energy are collected in \cite{LL}.

  In \cite{ABW}, Adiga et al. first introduced the skew energy of an oriented graph.
In particular, a sharp upper bound for the skew energy of an
oriented graph $G^\sigma$ was obtained in terms of order $n$ and
maximum degree $\Delta$ of $G$, that is
\begin{equation*}
\label{inque1}
\varepsilon_s(G^\sigma)\leq n\sqrt\Delta.
\end{equation*}

They pointed out that $\varepsilon_s(G^\sigma)= n\sqrt\Delta$ if and
only if $S(G^\sigma)^{T}S(G^\sigma)=\Delta I_n$, where $I_n$ is the
identity matrix of order $n$. In the following, we call an oriented
graph $G^\sigma$ on  $n$ vertices with maximum degree $\Delta$ an
\emph{optimum skew energy oriented graph} if and only if
$\varepsilon_s(G^\sigma)= n\sqrt\Delta$. At the same time the
orientation $\sigma$ is called the \emph{optimum orientation}.

  Which $k$-regular graphs $G$ on $n$ vertices have orientations $\sigma$
satisfying $\varepsilon_s(G^\sigma)= n\sqrt k$?

  In \cite{ABW}, Adiga et al. solved the problem for situations of $k=1$ and $k=2$.
A 1-regular graph on $n$
vertices has skew energy is $n$ if and only if $n$ is
even and the graph is $n/2$ copies of $K_2$, A 2-regular graph on
$n$ vertices has skew energy $n\sqrt{2}$ if and only if $n$ is
a multiple of 4 and it is a union of $n/4$ copies of $C_4$.

  Then, Tian determined the optimum orientations for hypercubes \cite{TG}.
In \cite{GX}, Gong and Xu
characterized all connected 3-regular oriented graphs with optimum skew
energy, they are the complete graph $K_4$ and the hypercube
$\mathbf{Q}_3$. In \cite{GXZ}, Gong et al. considered the connected 4-regular
oriented graphs with optimum skew energy, they found the possible connected
4-regular underlying graphs, which are
$U_3,U_n(n\geq5),Q_4,P_2\square K_4$ and the graph $G_1$ in Fig 2.2
of \cite{GXZ}. In \cite{CLL}, Chen et al. concentrated on the connected
4-regular oriented graphs with optimum skew energy, they found the
same five underlying graphs in another way.


  There is a connection between the above question and matrix theory.
A weighing matrix $W=W(n;k)$ (see \cite{CW}) is defined as a square $(0,\pm1)$ matrix with $k$
non-zero entries per row and column and inner product of distinct rows equal to zero.
Therefore, we know $WW^T=kI_n$. In fact the skew adjacency matrix of a optimum skew energy
oriented graph is a skew symmetric weighing matrix.

  In this paper, we continue to investigate the above question and determine all connected
5-regular oriented graphs of order $n$ with maximum skew energy.

\section{The possible underlying graphs of connected 5-regular oriented
graphs with optimum skew energy} \label{Underly-sec}

  If each pair of neighborhoods of a graph have even number of common vertices,
then this graph is called to have even neighborhood property.
As we shall see, every 5-regular oriented graph with optimum skew energy
has even neighborhood property. In this section, we characterize all connected
5-regular graphs of order $n$ with this property.

\begin{lemma}(Theorem 2.5 and Corollary 2.6 of \cite{ABW}) \label{le:c-1} Let
$G^\sigma$ be an oriented graph of $G$ on n vertices with maximum
degree $\Delta$. Then $\varepsilon_s(G^\sigma)\leqslant
n\sqrt\Delta$ with equality if and only if
$S(G^\sigma)^{T}S(G^\sigma)=\Delta I_n$, where $I_n$ is the identity
matrix of order $n$.
\end{lemma}



\begin{lemma} (Lemma 2.7 of \cite{ABW}) \label{le:c-4} Let
$S(G^\sigma)$ be the skew adjacency matrix of $G^\sigma$, if
$S(G^\sigma)^{T}S(G^\sigma)$\\$=kI_n$,
then $|N(u)\cap N(v)|$ is even for any distinct vertices $u$ and $v$
of $G^\sigma$.  In particular, if $G^\sigma$ is a
5-regular oriented graph of order $n$ satisfying $S(G^\sigma)^{T}S(G^\sigma)$ $=5I_n$, then
$|N(u)\cap N(v)|\in\{0,2,4\}$ for any two distinct vertices $u$ and $v$ of $G^\sigma$.
\end{lemma}




\begin{lemma} \label{le:c-5}
   Let $G_1$ and $G_2$ be vertex disjoint graphs. Then $G_1$ and $G_2$ have
optimum orientations if and only if the union of $G_1$ and $G_2$ does.
\end{lemma}

  The following observation allows one to restrict attention to connected graphs.
For any 5-regular oriented graph $G^\sigma$ of order $n$ with optimum
skew energy $n\sqrt5$,
we will discuss the following three cases for the underlying
graph $G$ of 5-regular oriented graphs $G^\sigma$. That is,
\begin{itemize}
\item  $G$ contains $K_4$ as an induced subgraph.
\item  $G$ contains a $K_3$ but not $K_4$ as an induced subgraph.
\item  $G$ contains no $K_3$.
\end{itemize}

  In order to determine the connected 5-regular graphs with the even neighborhood property, we first introduce firstly a
procedure table.  The first column of the table denotes the vertices $v_i$ $(i=1,2,\ldots,n)$
of the graph, and the other columns of the table are the neighbors of
the vertex $v_i$ for every row.
From the table, we will compare the neighborhoods of any two vertices in the
graph and determine the neighbors of each vertex in the
graph. Finally, we will find the underlying connected 5-regular graphs $G$ with the even
neighborhood property of optimum skew energy oriented graphs $G^\sigma$.

  Next we introduce an algorithm how to obtain the procedure table which will be used in
Theorems 2.4, 2.5 and 2.6 respectively.

{\bf Step 1.} Choose a vertex $v_1$ of graph $G$ and its neighbors $v_2,v_3,v_4,v_5,v_6$,
then show them in the table.

{\bf Step 2.} Preliminarily determine the neighbors of $v_2,v_3,v_4,v_5,v_6$ according
to whether the graph contains $K_4$ or $K_3$.

{\bf Step 3.} Compare the neighborhoods of any two vertices $v_i$ and $v_j$ in the table.
If $|N(v_i)\cap N(v_j)|$ is odd, then find new neighbor(s) for $v_i$ and $v_j$ from the table
such that $|N(v_i)\cap N(v_j)|$ is even under the conditions that the graph is 5-regular with
even neighborhood property and whether the graph contains $K_4$ or $K_3$.
If such new neighbor(s) do not exist for all pairs of vertices $v_i$ and $v_j$ in the table,
then go to Step 4.

{\bf Step 4.} If the degree of every vertex in the table is 5, stop; else choose a vertex $v$
with degree less than 5 and that is not adjacent to any vertex of the table.
Then add a new vertex $u$ to the table such that $u$ is adjacent to $v$, and return to Step 3.

  In the following, every step will be shown by an area enveloped in dotted line.
We will explain how to use it in the progress of obtaining $G_1$ and $G_2$ of
Theorem \ref{th:c-1}.

The graphs in Figure \ref{Fig-1} are the connected 5-regular graphs with the even neighborhood property that contains $K_4$ as an induced graph.







\begin{theorem} \label{th:c-1} Let $G^\sigma$ be a connected 5-regular
oriented graph on $n$ vertices with optimum skew energy $n\sqrt 5$.
If the underlying graph $G$ contains $K_4$ as an induced subgraph,
then $G$ must be one of the following graphs with the even neighborhood
property depicted in Figure \ref{Fig-1}.
\end{theorem}

\begin{proof}
  Since $G$ is a connected 5-regular graph which contains $K_4$ as an induced
subgraph, let $V(G[K_4])$ $=\{v_1,v_2,$ $v_3,v_4\}$ and the other neighbors are $v_5,v_6$.
Clearly any two vertices of $\{v_1,v_2,$ $v_3,v_4\}$ are adjacent.
This corresponds to Step 1 and Step 2 of Table 1. From
Lemma \ref{le:c-4}, we know $|N(v_1)\cap N(v_2)|$ is even. Thus, if $v_2v_5\in
E(G)$, we deduce $v_2v_6\in E(G)$. Similarly, if $v_3v_5\in
E(G)$ then $v_3v_6\in E(G)$, if $v_4v_5\in E(G)$ then $v_4v_6\in
E(G)$. Suppose $a=|\{v_5,v_6\}\cap N(v_2)|$, $b=|\{v_5,v_6\}\cap
N(v_3)|$, $c=|\{v_5,v_6\}\cap N(v_4)|$. Therefore the possible value
for $(a,b,c)$ is one of $\{(2,2,2), (2,2,0), (2,0,2), (0,2,2),$ $(2,0,$ $0), (0,2,0),$ $(0,0,2),(0,0,0)\}$. If we do not consider the
labeling order of $v_2,v_3,v_4$,
We only need to discuss four cases $(2,2,2), (2,2,0),(2,0,0), (0,0,0)$.
\begin{figure}[htbp]
\begin{centering}
\includegraphics[scale=0.8]{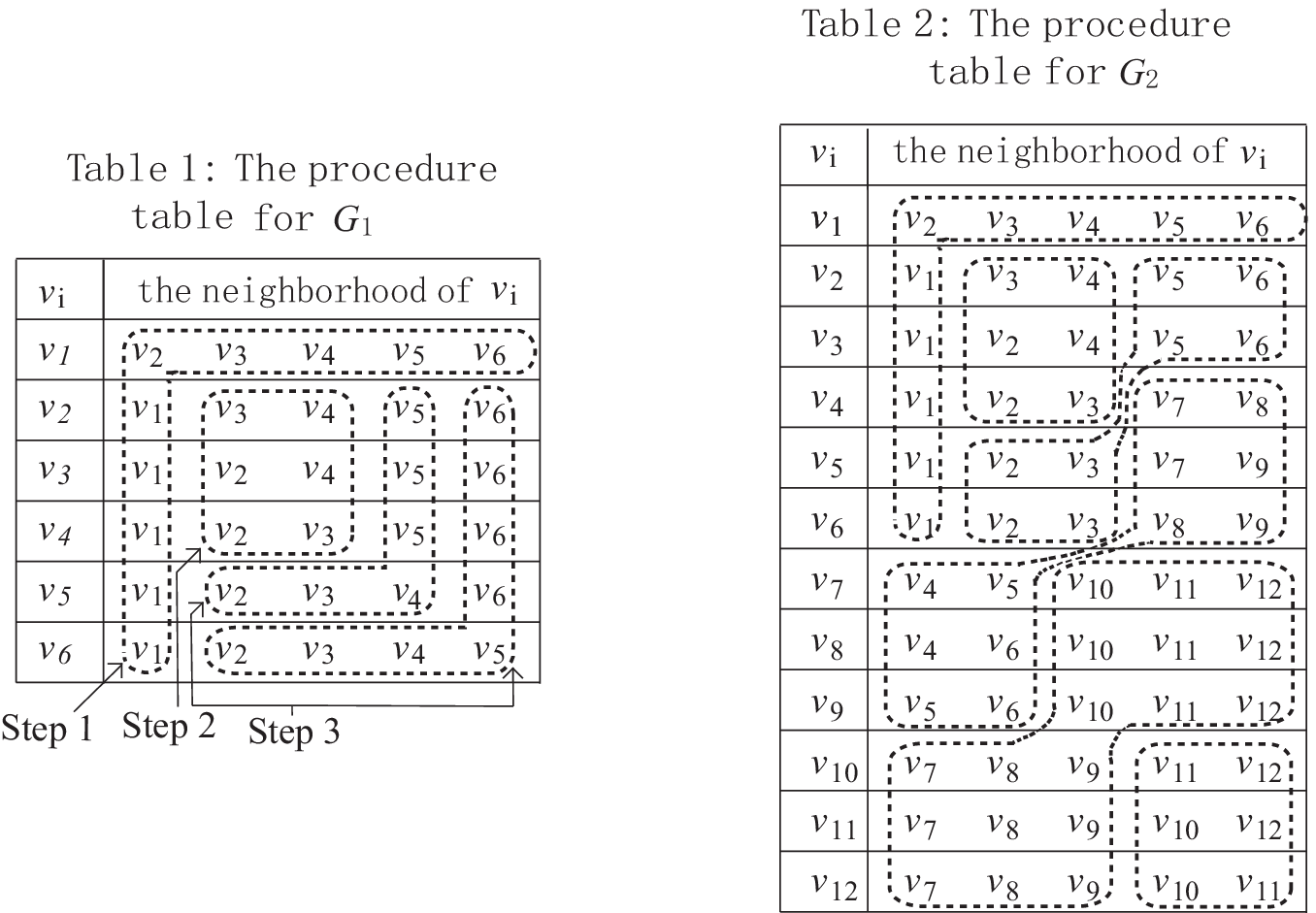}
\end{centering}
\end{figure}
\begin{figure}[htbp]
\begin{centering}
\includegraphics[scale=0.6]{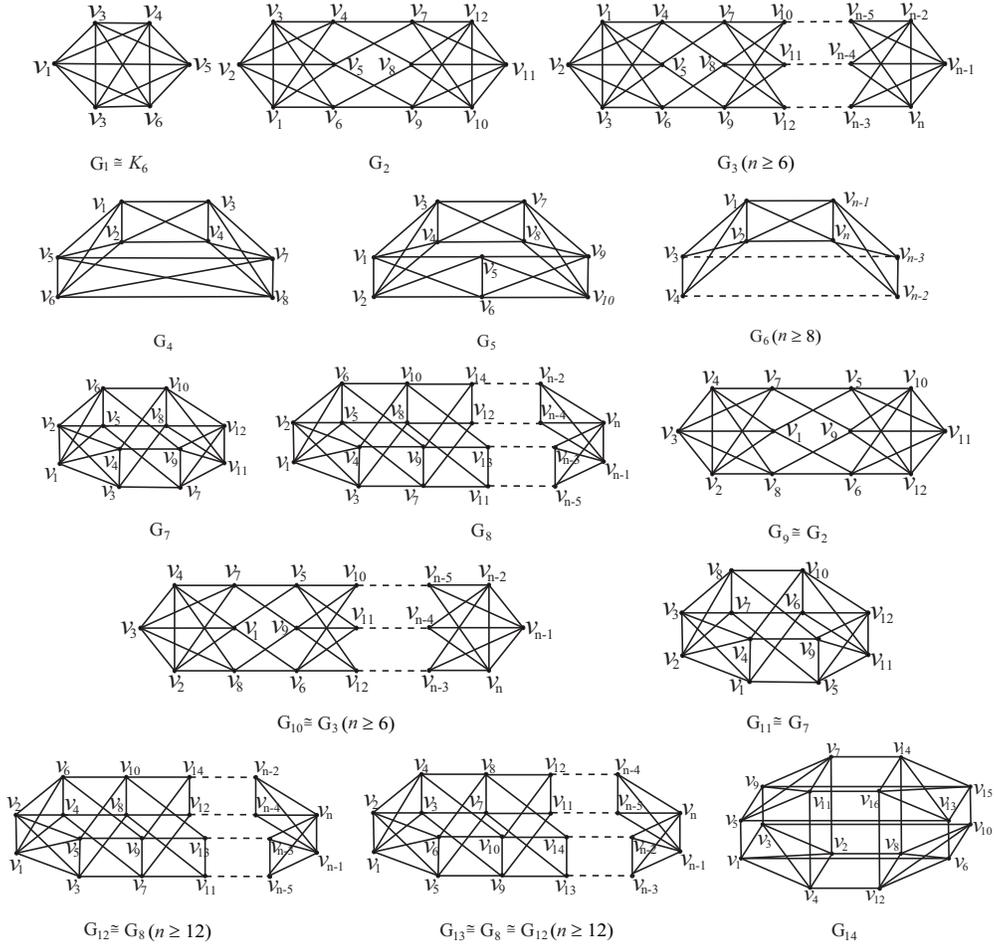}\\
\caption{All the possible underlying graphs for Theorem 2.4}\label{Fig-1}
\end{centering}
\end{figure}

\textit{Case 1. $(a,b,c)=(2,2,2).$}

  In this case, we know that $N(v_5)\supseteq\{v_1,v_2,v_3,v_4\}$ and
$N(v_1)=\{v_2,v_3,v_4,v_5,v_6\}$.
Applying Lemma \ref{le:c-4} to $|N(v_1)\cap N(v_5)|$, we have $v_6\in
N(v_5)$, Thus $N(v_5)=\{v_1,v_2,v_3,$ $v_4,v_6\}$ and $d(v_i)=5$, for
$i=1,2,3,4,5,6$. This corresponds to Step 3 of Table 1.

  Now, the degree of every vertex is 5 and the graph is connected,
so the graph $G_1$ is found.

   The resulting graph is $G_1\cong K_6$ with the even neighborhood property
depicted in Figure \ref{Fig-1}. We have shown the progress
of obtaining $G_1$ in Table 1.

\textit{Case 2. $(a,b,c)=(2,2,0)$.}

  In this case, $N(v_2)=\{v_1,v_3,v_4,v_5,v_6\}$,
$N(v_3)=\{v_1,v_2,v_4,v_5,v_6\}$,
$v_4v_5\notin E(G)$ and $v_4v_6\notin E(G)$.

  There must be two new vertices $v_7,v_8\in N(v_4)$.
Applying Lemma \ref{le:c-4} to $|N(v_7)$ $\cap N(v_1)|$,
either $v_7v_5\in E(G)$ or $v_7v_6\in E(G)$, and not both.
Without loss of generality,
let $v_7v_5\in E(G)$ and $v_7v_6\notin E(G)$. Then $N(v_6)\cap
N(v_4)\supseteq\{v_1,v_2,v_3\}$, so $v_6v_8\in E(G)$ from Lemma \ref{le:c-4}.
We deduce that $v_5v_6\notin E(G)$, otherwise $|N(v_1)\cap N(v_5)|=3$ and
$d(v_1)=d(v_5)=5$, which contradict to Lemma \ref{le:c-4}.
Since $d(v_6)=5$, there is another new neighbor $v_9$ of $v_6$.
Applying Lemma \ref{le:c-4} to $|N(v_5)\cap N(v_6)|$,
either $v_5v_8\in E(G)$ or $v_5v_9\in E(G)$. If $v_5v_8\in E(G)$,
then $|N(v_4)\cap N(v_5)|=5$, which contradicts to Lemma \ref{le:c-4}.
Thus $v_5v_8\notin E(G)$ and $v_5v_9\in E(G)$. So far, we have
$d(v_i)=5$ for $i=1,2,3,4,5,6$, and $d(v_j)=2$ for $j=7,8,9$. We
deduce that $v_7v_8\notin E(G), v_8v_9\notin E(G),v_7v_9\notin
E(G)$. Otherwise, without loss of generality, we assume $v_7v_8\in E(G)$, then
$|N(v_4)\cap N(v_7)|=1$. While $N(v_4)=\{v_1,v_2,v_3,v_7,v_8\},
d(v_1)=d(v_2)=d(v_3)=5$ and the graph is a simple graph, so there is a
contradiction. Similarly, $v_8v_9\notin E(G),v_7v_9\notin E(G)$.

  Now $v_7$ has three new neighbors, denoted by $v_{10},v_{11},v_{12}$.
From $|N(v_4)\cap N(v_{10})|$, we know $v_8v_{10}\in E(G)$.
Similarly we consider $|N(v_4)\cap N(v_{11})|,
|N(v_4)\cap N(v_{12})|, |N(v_5)\cap N(v_{10})|, |N(v_5)\cap
N(v_{11})|$ and $|N(v_5)\cap N(v_{12})|$, respectively. We deduce
$v_8v_{11}\in E(G),$\\
$v_8v_{12}\in E(G), v_9v_{10}\in E(G),
v_9v_{11}\in E(G)$ and $v_9v_{12}\in E(G)$, respectively.

  Next, we will discuss the following two cases.

\textit{Case 2.1.} There is a pair of adjacent vertices in
$\{v_{10},v_{11},v_{12}\}$.

  Without loss of generality, let $v_{10}v_{11}\in E(G)$, then
$N(v_7)=\{v_4,v_5,v_{10},v_{11},v_{12}\}$,
$N(v_7)\cap N(v_{10})\supseteq\{v_{11}\}$ and $d(v_4)=d(v_5)=5$. By Lemma \ref{le:c-4},
$v_{10}v_{12}\in E(G)$, similarly $v_{11}v_{12}\in E(G)$.
Since $d(v_i)=5$ for $i=1,2,\ldots,12$, the graph is $G_2$
with the even neighborhood property as shown in Figure \ref{Fig-1}.

  From Case 2 to here, every paragraph above corresponds a progress shown by an area
enveloped in dotted line in Table 2. In the following, all progress of proving Theorems
\ref{th:c-1}, \ref{th:c-2} and \ref{th:c-3} can be
shown in some similar procedure tables.

\textit{Case 2.2.} No vertices among $\{v_{10},v_{11} ,v_{12}\}$ are adjacent.

  In this case, there must be two new neighbors
$v_{13},v_{14}$ of $v_{10}$. By Lemma \ref{le:c-4}, exactly one of
$v_{11}v_{13}$ and $v_{12}v_{13}$ is an edge of $G$. Without loss of generality,
assume $v_{11}v_{13}$ is an edge of $G$.
By Lemma \ref{le:c-4}, $v_{12}v_{14}$ is an edge and $v_{11}v_{14}$ is not.
Thus there exists a new vertex $v_{15}$
that is adjacent to $v_{11}$, and by Lemma \ref{le:c-4} we have $v_{12}v_{15}$ is an edge.
Now, $d(v_{10})=d(v_{11})=d(v_{12})=5, d(v_{13})=d(v_{14})=d(v_{15})=2$.

   We deduce $v_{13}v_{14}\notin E(G), v_{13}v_{15}\notin E(G)$ and
$v_{14}v_{15}\notin E(G)$. Otherwise, without loss of generality, we assume
$v_{13}v_{14}\in E(G)$. Then $N(v_{10})=\{v_7,v_8,v_9,v_{13},v_{14}\}$,
$N(v_{10})\cap N(v_{13})=\{v_{14}\}$, and $d(v_7)=d(v_8)=d(v_9)=5$.
This is a contradiction to Lemma \ref{le:c-4}. Thus
there exist three new neighbors $v_{16},v_{17},v_{18}$ of
$v_{13}$. Since $|N(v_{13})\cap N(v_{14})|$ is even,
without of loss generality, we may assume $v_{14}v_{16}\in E(G)$. Similarly,
$v_{15}v_{16}\in E(G)$ (since $|N(v_{11})\cap N(v_{16})|$),
$v_{14}v_{17}\in E(G)$ ($|N(v_{10})\cap N(v_{17})|$),
$v_{15}v_{17}\in E(G)$ ($|N(v_{11})\cap N(v_{17})|$),
$v_{14}v_{18}\in E(G)$ ($|N(v_{10})\cap N(v_{18})|$),
$v_{15}v_{18}\in E(G)$ ($|N(v_{11})\cap N(v_{18})|$).

  Now we need to know the adjacency relationship of $v_{16},v_{17},v_{18}$,
which is similar to the adjacency relationship of $v_{10},v_{11},v_{12}$.
By further discussion, we obtain a class of graph
$G_3$ with order $n\geq 12$ and the even neighborhood property as shown in Figure \ref{Fig-1}.

\textit{Case 3.} $(a,b,c)=(2,0,0)$.

   By using a similar method to that used in Cases 1 and 2, we find underlying
graphs $G_4,G_5,G_6,$ $G_7,G_8$, Details are as follows.

   In this case $v_3v_5\notin E(G), v_3v_6\notin E(G),
v_4v_5\notin E(G), v_4v_6\notin E(G)$. Let the new neighbors of
$v_3$ be $v_7,v_8$. Applying Lemma \ref{le:c-4} to $|N(v_5)\cap N(v_2)|$,
we obtain $v_5v_6\in E(G)$. Next we discuss the following three cases.

\textit{Case 3.1.} $v_4v_7\in E(G)$ and $v_5v_7\in E(G)$.

  We deduce $v_4v_8\in E(G)$ ($|N(v_3)\cap N(v_4)|$).
Similarly $v_7v_8\in E(G)$ ($|N(v_4)\cap N(v_7)|$),
$v_5v_8\in E(G)$ ($|N(v_4)\cap N(v_5)|$),
$v_6v_7\in E(G)$ ($|N(v_1)\cap N(v_7)|$),
$v_6v_8\in E(G)$ ($|N(v_1)\cap N(v_8)|$).
Now $d(v_i)=5$, for $i=1,2,\ldots,8$. The graph is $G_4$ with the even neighborhood
property as shown in Figure \ref{Fig-1}.

\textit{Case 3.2.} $v_4v_7\in E(G)$ and $v_5v_7\notin E(G)$.

We know $v_4v_8\in E(G)$ ($|N(v_3)\cap N(v_4)|)$,
$v_7v_8\in E(G)$ ($|N(v_4)\cap N(v_7)|$),
$v_5v_8\notin E(G)$ from $v_5v_7\notin E(G)$ and $|N(v_4)\cap N(v_5)|$.
There are two new neighbors $v_9,v_{10}$ of $v_5$ and $v_6v_9\in E(G)$
($|N(v_1)\cap N(v_9)|$), $v_6v_{10}\in E(G)$ ($|N(v_1)\cap N(v_{10})|$).

\textit{Case 3.2.1.} $\{v_7v_9,v_7v_{10},v_8v_9,v_8v_{10}\}\cap
E(G)\neq\emptyset$.

Without loss of generality, we assume that $v_7v_9\in E(G)$.
Then $v_7v_{10}\in E(G)$ ($|N(v_6)\cap N(v_7)|$),
$v_8v_9\in E(G)$ ($|N(v_3)\cap N(v_9)|$),
$v_8v_{10}\in E(G)$ ($|N(v_6)\cap N(v_8)|$),
$v_9v_{10}\in E(G)$ ($|N(v_5)\cap N(v_9)|$).
The graph $G_5$ is found as shown in Figure \ref{Fig-1}.

\textit{Case 3.2.2.} $\{v_7v_9,v_7v_{10},v_8v_9,v_8v_{10}\}\cap
E(G)=\emptyset$.

  There are two new neighbors $v_{11}, v_{12}$ of $v_7$.
Then $v_8v_{11}\in E(G)$ ($|N(v_3)\cap N(v_{11})|$),
$v_8v_{12}\in E(G)$ ($|N(v_3)\cap N(v_{12})|$),
$v_{11}v_{12}\in E(G)$ ($|N(v_7)\cap N(v_{11})|$),
$v_9v_{10}\in E(G)$ ($|N(v_5)\cap N(v_9)|$).

 Next we should discuss
whether $\{v_9v_{11},v_9v_{12},v_{10}v_{11},$ $v_{10}v_{12}\}\cap
E(G)=\emptyset$ or not. This discussion is similar to the discussion
that whether $\{v_7v_9,v_7v_{10},$ $v_8v_9,v_8v_{10}\}\cap
E(G)=\emptyset$ or not. With further discussion, we obtain a class
of graph which is isomorphic to $G_6$ with order $n\geq 8$ in Figure \ref{Fig-1}.

  \textit{Case 3.3.} $v_4v_7\notin E(G)$.

  We deduce $v_4v_8\notin E(G)$ ($|N(v_3)\cap N(v_4)|$).
Then there exist two new neighbors $v_9,v_{10}$ of $v_4$.
Now $|N(v_7)\cap N(v_1)|$, so either $v_5v_7\in E(G)$ or
$v_6v_7\in E(G)$, and not both. While $N(v_5)\supseteq\{v_1,v_2,v_6\}$
and $N(v_6)\supseteq\{v_1,v_2,v_5\}$, without loss of generality, we
assume $v_5v_7\in E(G)$ and $v_6v_7\notin E(G)$.
Then $v_5v_8\in E(G)$ ($|N(v_3)\cap N(v_5)|$).
$v_6v_7\notin E(G)$ and $|N(v_5)\cap N(v_6)|$ imply $v_6v_8\notin E(G)$.
$|N(v_4)\cap N(v_7)|$ implies either $v_7v_9\in E(G)$ or $v_7v_{10}\in E(G)$, and not both.
Without loss of generality, we assume $v_7v_9\in E(G)$ and $v_7v_{10}\notin E(G)$.

  Similarly, we know either $v_8v_9\in E(G)$ or $v_8v_{10}\in E(G)$. But if $v_8v_9\in E(G)$,
$N(v_3)\cap N(v_9)=\{v_4,v_7,v_8\}$ and for the other neighbors
$v_1,v_2$ of $v_3$, $d(v_1)=d(v_2)=5$.
So $v_8v_9\notin E(G)$ and $v_8v_{10}\in E(G)$.

  Now $v_6v_9\in E(G)$ ($|N(v_1)\cap N(v_9)|$), similarly $v_6v_{10}\in E(G)$.
We deduce $v_7v_8\notin E(G)$ ($|N(v_5)\cap N(v_7)|$), coupled with
$v_7v_{10}\notin E(G)$, so there are two new neighbors $v_{11},v_{12}$ of $v_7$.
Then $v_8v_{11}\in E(G)$ ($|N(v_3)\cap N(v_{11})|$), similarly $v_8v_{12}\in E(G)$.
Next we consider whether
$\{v_9v_{11},v_9v_{12},v_{10}v_{11},v_{10}v_{12}\}\cap E(G)=\emptyset$ or not.

  Similar to Cases 3.2.1 and 3.2.2, we obtain the graphs $G_7$ and $G_8$
depicted in Figure \ref{Fig-1}.

\textit{Case 4.} $(a,b,c)=(0,0,0)$.

  Similarly, we find the underlying graphs $G_9,G_{10},G_{11},$ $G_{12},G_{13},G_{14}$
with the even neighbors property. Details are as follows

  In this situation we have $v_2v_5\notin E(G)$, $v_3v_5\notin E(G)$, $v_4v_5\notin E(G)$.
So $v_5v_6\notin E(G)$, otherwise $|N(v_1)\cap N(v_5)|=1$,
which contradicts to Lemma \ref{le:c-4}.

  Thus there are two new neighbors $v_7,v_8$ of $v_2$.
$|N(v_2)\cap N(v_5)|$ implies $v_5v_7\in E(G)$ or $v_5v_8\in E(G)$. Without loss of
generality, let $v_5v_7\in E(G)$ and $v_5v_8\notin E(G)$. Now we are
not sure whether $v_3v_7\in E(G)$, $v_4v_7\in E(G)$ or not,
We note that $v_3$ and $v_4$ are symmetric, so we only need to discuss the three cases
according to $|\{v_3v_7,v_4v_7\}\cap E(G)|$.

\textit{Case 4.1.} $|\{v_3,v_4\}\cap N(v_7)|=2$.

  In this case, $v_3v_7\in E(G)$ and $v_4v_7\in E(G)$.
Then $v_3v_8\in E(G)$ ($|N(v_2)\cap N(v_3)|$),
$v_4v_8\in E(G)$ ($|N(v_3)\cap N(v_4)|$).
$|N(v_1)\cap N(v_8)|$ implies $v_5v_8\in E(G)$ or $v_6v_8\in E(G)$, and not both.

   We deduce $v_5v_8\notin E(G)$, otherwise $|N(v_4)\cap N(v_5)|$ contradicts to
Lemma \ref{le:c-4}, so $v_5v_8\notin E(G)$ and $v_6v_8\in E(G)$.
Then $v_6v_7\notin E(G)$ ($|N(v_4)\cap N(v_6)|$), coupled with $v_5v_8\notin E(G)$ and
$|N(v_7)\cap N(v_8)|$, so $v_7$ and $v_8$ have another common neighbor,
denote it by $v_9$.
Then $v_5v_6\notin E(G)$ ($|N(v_1)\cap N(v_5)|$),
$v_5v_9\notin E(G)$ ($|N(v_5)\cap N(v_7)|$).
So $v_5$ has three new neighbors, denoted by $v_{10},v_{11},v_{12}$.
Then $v_6v_{10}\in E(G)$ ($|N(v_1)\cap N(v_{10})|$),
$v_9v_{10}\in E(G)$ ($|N(v_7)\cap N(v_{10})|$).
Similarly $v_6v_{11}\in E(G)$($|N(v_1)\cap N(v_{11})|$),
$v_9v_{11}\in E(G)$($|N(v_7)\cap N(v_{11})|$),
$v_6v_{12}\in E(G)$($|N(v_1)\cap N(v_{12})|$),
$v_9v_{12}\in E(G)$($|N(v_7)\cap N(v_{12})|$). Now
$N(v_{10})\supseteq\{v_5,v_6,v_9\}$,
$N(v_{11})\supseteq\{v_5,v_6,v_9\}$,
$N(v_{12})\supseteq\{v_5,v_6,v_9\}$.

\textit{Case 4.1.1.} $\{v_{10}v_{11},v_{10}v_{12},v_{11}v_{12}\}\cap
E(G)\neq\emptyset$.

Without loss of generality, let $v_{10}v_{11}\in E(G)$.
Then $v_{10}v_{12}\in E(G)$ ($|N(v_9)\cap N(v_{10})|$),
$v_{11}v_{12}\in E(G)$ ($|N(v_{10})\cap N(v_{11})|$).
Now $d(v_i)=5$ for $i=1,2,...,12$, then the graph is $G_9$ as shown in Figure \ref{Fig-1}.

\textit{Case 4.1.2.} $\{v_{10}v_{11},v_{10}v_{12},v_{11}v_{12}\}\cap
E(G)=\emptyset$.

  Similar to Case 4.1.1, we obtain the graph $G_{10}$ as shown in Figure \ref{Fig-1}.

\textit{Case 4.2.} $|\{v_3,v_4\}\cap N(v_7)|=1$.

Without loss of generality, let $v_3v_7\in E(G)$ and $v_4v_7\notin E(G)$.
Then $v_3v_8\in E(G)$ ($|N(v_2)\cap N(v_3)|$),
$|N(v_1)\cap N(v_7)|$ implies either $v_4v_7\in E(G)$ or $v_6v_7\in E(G)$, and not both.
Coupled with $v_4v_7\notin E(G)$, we conclude $v_6v_7\in E(G)$.
$|N(v_2)\cap N(v_7)|$ implies either $v_4v_7\in E(G)$ or $v_7v_8\in E(G)$, and not both.
Coupled with $v_4v_7\notin E(G)$, we conclude that $v_7v_8\in E(G)$.
Then $v_4v_8\notin E(G)$ ($|N(v_3)\cap N(v_4)|$).
There are two neighbors $v_9,v_{10}$ of $v_4$.
$|N(v_4)\cap N(v_5)|$ implies $v_5v_9\in E(G)$ or $v_5v_{10}\in E(G)$, and not both.
Without loss of generality, let $v_5v_9\in E(G)$ and $v_5v_{10}\notin E(G)$.
$|N(v_4)\cap N(v_6)|$ implies either $v_6v_9\in E(G)$ or $v_6v_{10}\in E(G)$, and not both.

   We claim that $v_6v_9\notin E(G)$. Otherwise $N(v_1)\cap N(v_9)=\{v_4,v_5,v_6\}$
and for the other neighbors $v_2$ and $v_3$ of $v_1$, $d(v_2)=d(v_3)=5$.
This contradicts to Lemma \ref{le:c-4}, so $v_6v_9\notin E(G)$ and $v_6v_{10}\in E(G)$.

  $|N(v_2)\cap N(v_9)|$ implies either $v_7v_9\in E(G)$ or $v_8v_9\in E(G)$, and not both.
But $d(v_7)=5$, so $v_8v_9\in E(G)$.
Similarly $v_8v_{10}\in E(G)$ ($|N(v_2)\cap N(v_{10})|$),
$v_5v_6\notin E(G)$ ($|N(v_1)\cap N(v_5)|$),
$v_5v_{10}\notin E(G)$ ($|N(v_4)\cap N(v_5)|$).
There are two new neighbors $v_{11},v_{12}$ of $v_5$.
Then $v_6v_{11}\in E(G)$ ($|N(v_1)\cap N(v_{11})|$).
Similarly $v_6v_{12}\in E(G)$, $v_9v_{10}\notin E(G)$ ($|N(v_8)\cap N(v_9)|$).
So far, $d(v_9)=d(v_{10})=3$, $d(v_{11})=d(v_{12})=2$.

\textit{Case 4.2.1.}
$\{v_9v_{11},v_9v_{12},v_{10}v_{11},v_{10}v_{12}\}\cap
E(G)\neq\emptyset$.

Without loss of generality, let $v_9v_{11}\in E(G)$.
Then $v_9v_{12}\in E(G)$ ($|N(v_6)\cap N(v_9)|$),
$v_{10}v_{11}\in E(G)$ ($|N(v_4)\cap N(v_{11})|$).
Similarly $v_{10}v_{12}\in E(G)$ ($N(v_4)\cap N(v_{12})$).
Then $v_{11}v_{12}\in E(G)$ ($|N(v_6)\cap N(v_{11})|$).
Since $d(v_i)=5$ for $i=1,2,...,12$, the graph is $G_{11}$
as shown in Figure \ref{Fig-1}.

\textit{Case 4.2.2.}
$\{v_9v_{11},v_9v_{12},v_{10}v_{11},v_{10}v_{12}\}\cap
E(G)=\emptyset$.

  Similarly, we obtain the graph $G_{12}$ with order $n\geq 12$ as shown in Figure \ref{Fig-1}.

\textit{Case 4.3.}  $|\{v_3,v_4\}\cap N(v_7)|=0$.

  In this case, $v_3v_7\notin E(G)$.
Then $v_3v_8\notin E(G)$ ($|N(v_2)\cap N(v_3)|$),
so there are another two new neighbors $v_9,v_{10}$ of $v_3$.
$|N(v_3)\cap N(v_5)|$ implies either $v_5v_9\in E(G)$ or $v_5v_{10}\in E(G)$,
and not both. Without loss of generality,
let $v_5v_9\in E(G)$ and $v_5v_{10}\notin E(G)$.
$v_4v_7\notin E(G)$ and $|N(v_2)\cap N(v_4)|$ imply $v_4v_8\notin
E(G)$. So far, we do not know whether $v_4v_9\in E(G)$ or not.
Next we will discuss the following two cases.

   \textit{Case 4.3.1.} $v_4v_9\in E(G)$.

  In this case, $v_4v_{10}\in E(G)$ ($|N(v_3)\cap N(v_4)|$),
$v_6v_9\in E(G)$ ($|N(v_1)\cap N(v_9)|$), and
$v_9v_{10}\in E(G)$ ($|N(v_3)\cap N(v_9)|$).
$|N(v_3)\cap N(v_7)|$ implies either $v_7v_9\in E(G)$ or $v_7v_{10}\in E(G)$,
and not both. Since $d(v_9)=5$, $v_7v_{10}\in E(G)$ and $v_7v_9\notin E(G)$.
$|N(v_3)\cap N(v_8)|$ implies either $v_8v_9\in E(G)$ or $v_8v_{10}\in E(G)$,
and not both. Since $d(v_9)=5$, $v_8v_9\notin E(G)$ and $v_8v_{10}\in E(G)$.
$|N(v_1)\cap N(v_8)|$ implies either $v_5v_8\in E(G)$ or $v_6v_8\in E(G)$,
and not both.

   We claim that $v_5v_8\notin E(G)$. Otherwise if $v_5v_8\in E(G)$,
then $|N(v_2)\cap N(v_5)|$ implies either $v_3v_5\in E(G)$ or
$v_4v_5\in E(G)$, and not both. But $d(v_3)=d(v_4)=5$, this is a
contradiction.

  So $v_5v_8\notin E(G)$, $v_6v_8\in E(G)$. So far,
$d(v_5)=d(v_6)=d(v_7)=d(v_8)=3$,
$|N(v_1)\cap N(v_5)|$ implies $v_5v_6\notin E(G)$. Coupled with $v_5v_8\notin E(G)$,
there are two new neighbors $v_{11},v_{12}$ of $v_5$.
Then $v_6v_{11}\in E(G)$ ($|N(v_1)\cap N(v_{11})|$). Similarly $v_6v_{12}\in E(G)$.
So far, $d(v_7)=d(v_8)=3$, $d(v_{11})=d(v_{12})=2$. This situation is similar to
Cases 4.2.1 and 4.2.2, and we obtain the class of graphs
$G_{13}$ as shown in Figure \ref{Fig-1}.

\textit{Case 4.3.2.} $v_4v_9\notin E(G)$.

 Similarly, we obtain the graph $G_{14}$ as shown in Figure \ref{Fig-1}.
\end{proof}

The graphs in Figure \ref{Fig-2} are the connected 5-regular
graphs with the even neighborhood
property that contains at least a triangle and
contains no $K_4$ as an induced subgraph.

\begin{theorem}\label{th:c-2} Let $G^\sigma$ be a connected 5-regular
optimum skew energy oriented graph on $n$ vertices,
if the underlying graph $G$ contains a triangle and
contains no $K_4$ as an induced subgraph, then $G$ must be one graph
depicted in Figure \ref{Fig-2}.
\end{theorem}

\begin{figure}[htbp]
\begin{centering}
\includegraphics[scale=0.6]{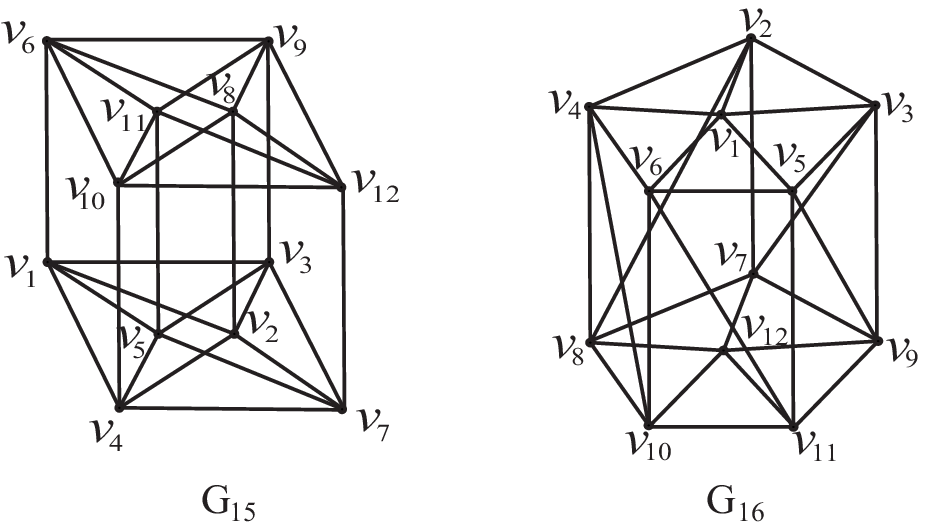}\\
\caption{All the possible underlying graphs for Theorem \ref{th:c-2}.}\label{Fig-2}
\end{centering}
\end{figure}


\begin{proof}
  We assume the three vertices of a triangle in $G$ are
$v_1,v_2,v_3$. Let $v_4,v_5,v_6$ be the other three neighbors of $v_1$.
By Lemma \ref{le:c-4}, we know
$v_2$ must be adjacent to at least one vertex of $v_4,v_5,v_6$.
Without loss of generality, let $v_2v_4\in E(G)$.
Since $N(v_1)\cap N(v_3)\supseteq \{v_2\}$, $v_3$ must be adjacent to at least one
vertex of $v_4,v_5,v_6$. We deduce $v_3v_4\notin E(G)$, otherwise the graph has an
induced subgraph $K_4$ with the induced vertices $v_1,v_2,v_3,v_4$.
Then $v_3v_5\in E(G)$ or $v_3v_6\in E(G)$, without loss of generality,
we assume that $v_3v_5\in E(G)$ and $v_3v_6\notin E(G)$.
$|N(v_1)\cap N(v_4)|$ implies either $v_4v_5\in E(G)$ or $v_4v_6\in
E(G)$. Next we will discuss the cases
according to whether $v_4v_5\in E(G)$ or not.

\textit{Case 1.} $v_4v_5\in E(G)$.

  We deduce $v_2v_5\notin E(G)$ and $v_2v_6\notin E(G)$,
otherwise if $v_2v_5\in E(G)$, then $v_2v_6\in E(G)$ ($|N(v_1)\cap N(v_2)|$).
Thus the graph has an induced subgraph $K_4$ with
vertices $v_1,v_2,v_3,$ $v_5$, so $v_2v_5\notin E(G)$.
Then $v_2v_6\notin E(G)$ ($|N(v_1)\cap N(v_2)|$).

   Thus there are two new neighbors $v_7, v_8$ of $v_2$.
$|N(v_2)\cap N(v_3)|$ and $v_3v_4\notin E(G)$ imply either
$v_3v_7\in E(G)$ or $v_3v_8\in E(G)$, and not both.
Without loss of generality, let $v_3v_7\in E(G)$ and $v_3v_8\notin E(G)$.
Then $v_3v_6\notin E(G)$ ($|N(v_1)\cap N(v_3)|$), and thus there is a
new neighbors $v_9$ of $v_3$. $|N(v_3)\cap N(v_4)|$ implies either
$v_4v_7\in E(G)$ or $v_4v_9\in E(G)$, and not both.

  We deduce $v_4v_9\notin E(G)$. Otherwise if $v_4v_9\in E(G)$ and $v_4v_7\notin E(G)$,
then $v_4v_8\in E(G)$ ($|N(v_2)\cap N(v_4)|$),
$|N(v_2)\cap N(v_5)|$ implies either $v_5v_7\in E(G)$ or $v_5v_8\in E(G)$, not
both.

  If $v_5v_7\notin E(G)$ and $v_5v_8\in E(G)$, then
$v_5v_9\in E(G)$ ($|N(v_3)\cap N(v_5)|$).
We deduce $d(v_4)=d(v_5)=5$ and $|N(v_4)\cap N(v_5)|=3$.
This contradicts to Lemma \ref{le:c-4}.

  If $v_5v_7\in E(G)$ and $v_5v_8\notin E(G)$,
then $v_5v_9\in E(G)$ ($|N(v_4)\cap N(v_5)|$).
We deduce $d(v_3)=d(v_5)=5$ and $|N(v_3)\cap N(v_5)|=3$.
This is a contradiction to Lemma \ref{le:c-4}.

  So $v_5v_7\notin E(G)$ and $v_5v_8\notin E(G)$.
Therefore $v_4v_9\notin E(G)$ and $v_4v_7\in E(G)$.

  We deduce $v_4v_6\notin E(G)$ ($|N(v_1)\cap N(v_4)|$),
$v_4v_8\notin E(G)$ ($|N(v_2)\cap N(v_4)|$) and
$v_4v_9\notin E(G)$ ($|N(v_3)\cap N(v_4)|$),
so there is a new neighbor $v_{10}$ of $v_4$.
$|N(v_2)\cap N(v_5)|$ implies either $v_5v_7\in E(G)$ or $v_5v_8\in E(G)$,
and not both.
We claim $v_5v_8\notin E(G)$ and $v_5v_7\in E(G)$. Otherwise if $v_5v_8\in E(G)$ and
$v_5v_7\notin E(G)$, then $v_5v_9\in E(G)$ ($|N(v_3)\cap N(v_5)|$).
But $d(v_4)=d(v_5)=5$ and $|N(v_4)\cap N(v_5)|=1$,
this contradicts to Lemma \ref{le:c-4}.
So $v_5v_8\notin E(G)$ and $v_5v_7\in E(G)$.

  We deduce $v_5v_6\notin E(G)$ ($|N(v_1)\cap N(v_5)|$),
$v_5v_8\notin E(G)$ ($|N(v_2)\cap N(v_5)|$),
$v_5v_9\notin E(G)$ ($|N(v_3)\cap N(v_5)|$),
and $v_5v_{10}\notin E(G)$ ($|N(v_4)\cap N(v_5)|$).
So there is a new neighbor $v_{11}$ of $v_5$.
Then $v_6v_{11}\in E(G)$ ($|N(v_1)\cap N(v_{11})|$).
$|N(v_2)\cap N(v_6)|$ implies $v_6v_7\in E(G)$ or $v_6v_8\in E(G)$, and not both.
But if $v_6v_7\in E(G)$, then $|N(v_1)\cap N(v_7)|=5$, so
$v_6v_7\notin E(G)$ and $v_6v_8\in E(G)$.
Then $v_6v_9\in E(G)$ ($|N(v_3)\cap N(v_6)|$),
$v_6v_{10}\in E(G)$ ($|N(v_4)\cap N(v_6)|$),
$v_7v_8\notin E(G)$ ($|N(v_2)\cap N(v_7)|$),
$v_7v_9\notin E(G)$ ($|N(v_3)\cap N(v_7)|$),
$v_7v_{10}\notin E(G)$ ($|N(v_6)\cap N(v_7)|$),
and $v_7v_{11}\notin E(G)$ ($|N(v_6)\cap N(v_7)|$).
So there is a new neighbor $v_{12}$ of $v_7$.
Then $v_8v_{12}\in E(G)$ ($|N(v_2)\cap N(v_{12})|$),
$v_9v_{12}\in E(G)$ ($|N(v_3)\cap N(v_{12})|$),
$v_{10}v_{12}\in E(G)$ ($|N(v_4)\cap N(v_{12})|$),
$v_{11}v_{12}\in E(G)$ ($|N(v_5)\cap N(v_{12})|$),
$v_8v_9\in E(G)$ ($|N(v_3)\cap N(v_8)|$),
$v_9v_{11}\in E(G)$ ($|N(v_5)\cap N(v_9)|$),
$v_8v_{10}\in E(G)$ ($|N(v_4)\cap N(v_8)|$),
and $v_{10}v_{11}\in E(G)$ ($|N(v_9)\cap N(v_{10})|$).
Since $d(v_i)=5$ for $i=1,2,...,12$, the graph is $G_{15}$ as shown in Figure \ref{Fig-2}.

\textit{Case 2.} $v_4v_5\notin E(G)$.

  In this case we have $v_4v_6\in E(G)$.
$|N(v_1)\cap N(v_5)|$ implies either $v_2v_5\in E(G)$ or $v_5v_6\in E(G)$, and not both.
But if $v_2v_5\in E(G)$, then there is an induced subgraph $K_4$ with the
induced vertices $v_1,v_2,v_3,v_5$.
So $v_2v_5\notin E(G)$ and $v_5v_6\in E(G)$.
Then $v_2v_6\notin E(G)$ ($|N(v_1)\cap N(v_2)|$).
So there are two new neighbors $v_7,v_8$ of $v_2$.
$|N(v_2)\cap N(v_3)|$ implies either $v_3v_7\in E(G)$ or $v_3v_8\in E(G)$,
and not both. Without loss of generality,
let $v_3v_7\in E(G)$ and $v_3v_8\notin E(G)$.
Now $v_3v_4\notin E(G)$, otherwise there is an induced
subgraph $K_4$ with induce vertices $v_1,v_2,v_3,v_4$.
Then $v_3v_6\notin E(G)$ ($|N(v_1)\cap N(v_3)|$),
$v_3v_8\notin E(G)$ ($|N(v_2)\cap N(v_3)|$).
So there is a new neighbor $v_9$ of $v_3$.
$|N(v_2)\cap N(v_4)|$ implies either $v_4v_7\in E(G)$ or $v_4v_8\in E(G)$,
and not both.

  We claim $v_4v_7\notin E(G)$ and $v_4v_8\in E(G)$. Otherwise if $v_4v_7\in E(G)$,
$v_4v_9\in E(G)$ ($|N(v_3)\cap N(v_4)|$),
$|N(v_7)\cap N(v_1)|$ implies either $v_5v_7\in E(G)$ or $v_6v_7\in E(G)$.

  If $v_5v_7\in E(G)$, then $v_5v_9\in E(G)$ ($|N(v_4)\cap N(v_5)|$).
But $|N(v_3)\cap N(v_5)|=3$ and $d(v_3)=d(v_5)=5$,
this contradicts to Lemma \ref{le:c-4}.

  If $v_6v_7\in E(G)$, similarly we obtain $v_6v_9\in E(G)$,
$|N(v_4)\cap N(v_6)|=3$ and $d(v_4)=d(v_6)=5$, this
contradicts to Lemma \ref{le:c-4}.

   Thus we know $v_4v_7\notin E(G)$ and $v_4v_8\in E(G)$.
Then $v_4v_5\notin E(G)$ ($|N(v_1)\cap N(v_4)|$),
$|N(v_3)\cap N(v_4)|$ implies $v_4v_7\notin E(G)$ and $v_4v_9\notin E(G)$,
so there is a new neighbor $v_{10}$ of $v_4$.
Then $v_7v_8\in E(G)$ ($|N(v_2)\cap N(v_7)|$).
$|N(v_3)\cap N(v_5)|$ implies either $v_5v_7\in E(G)$ or $v_5v_9\in E(G)$,
and not both.
We claim $v_5v_7\notin E(G)$, otherwise $v_5v_7\in E(G)$ and
$v_6v_7\in E(G)$ ($|N(v_1)\cap N(v_7)|$). But $d(v_4)=d(v_7)=5$ and
$|N(v_4)\cap N(v_7)|=3$, this contradicts to Lemma \ref{le:c-4}.
So $v_5v_7\notin E(G)$ and $v_5v_9\in E(G)$.

  Now $v_5v_8\notin E(G)$ ($|N(v_4)\cap N(v_5)|$),
$v_5v_{10}\notin E(G)$ ($|N(v_4)\cap N(v_5)|$).
So there is a new neighbor $v_{11}$ of $v_5$.
Then $v_7v_9\in E(G)$ ($|N(v_3)\cap N(v_9)|$),
$v_6v_{10}\in E(G)$ ($|N(v_1)\cap N(v_{10})|$),
$v_6v_{11}\in E(G)$ ($|N(v_1)\cap N(v_{11})|$),
$v_7v_{10}\notin E(G)$ ($|N(v_6)\cap N(v_7)|$),
$v_7v_{11}\notin E(G)$ ($|N(v_6)\cap N(v_7)|$).
So there is a new neighbor $v_{12}$ of $v_7$.
Then $v_8v_{12}\in E(G)$ ($|N(v_2)\cap N(v_{12})|$),
$v_9v_{12}\in E(G)$ ($|N(v_3)\cap N(v_{12})|$),
$v_{10}v_{12}\in E(G)$ ($|N(v_4)\cap N(v_{12})|$),
$v_{11}v_{12}\in E(G)$ ($|N(v_5)\cap N(v_{12})|$),
$v_8v_{10}\in E(G)$ ($|N(v_4)\cap N(v_8)|$),
and $v_9v_{11}\in E(G)$ ($|N(v_5)\cap N(v_9)|$),
$v_{10}v_{11}\in E(G)$ ($|N(v_9)\cap N(v_{10})|$).
Since $d(v_i)=5$ for $i=1,2,...,12$, the graph is $G_{16}$
as shown in Figure \ref{Fig-2}.

\end{proof}

The graphs in Figure \ref{Fig-3} are the connected 5-regular
graphs with the even neighborhood
property that contains no $K_3$ as an induced subgraph.

\noindent\begin{theorem}\label{th:c-3} Let $G^\sigma$ be a connected 5-regular
optimum skew energy oriented graph on $n$ vertices.
If the underlying graph $G$ contains no triangle, then $G$ is
one of the graphs depicted in Figure \ref{Fig-3}.
\end{theorem}

\begin{figure}[htbp]
\begin{centering}
\includegraphics[scale=0.6]{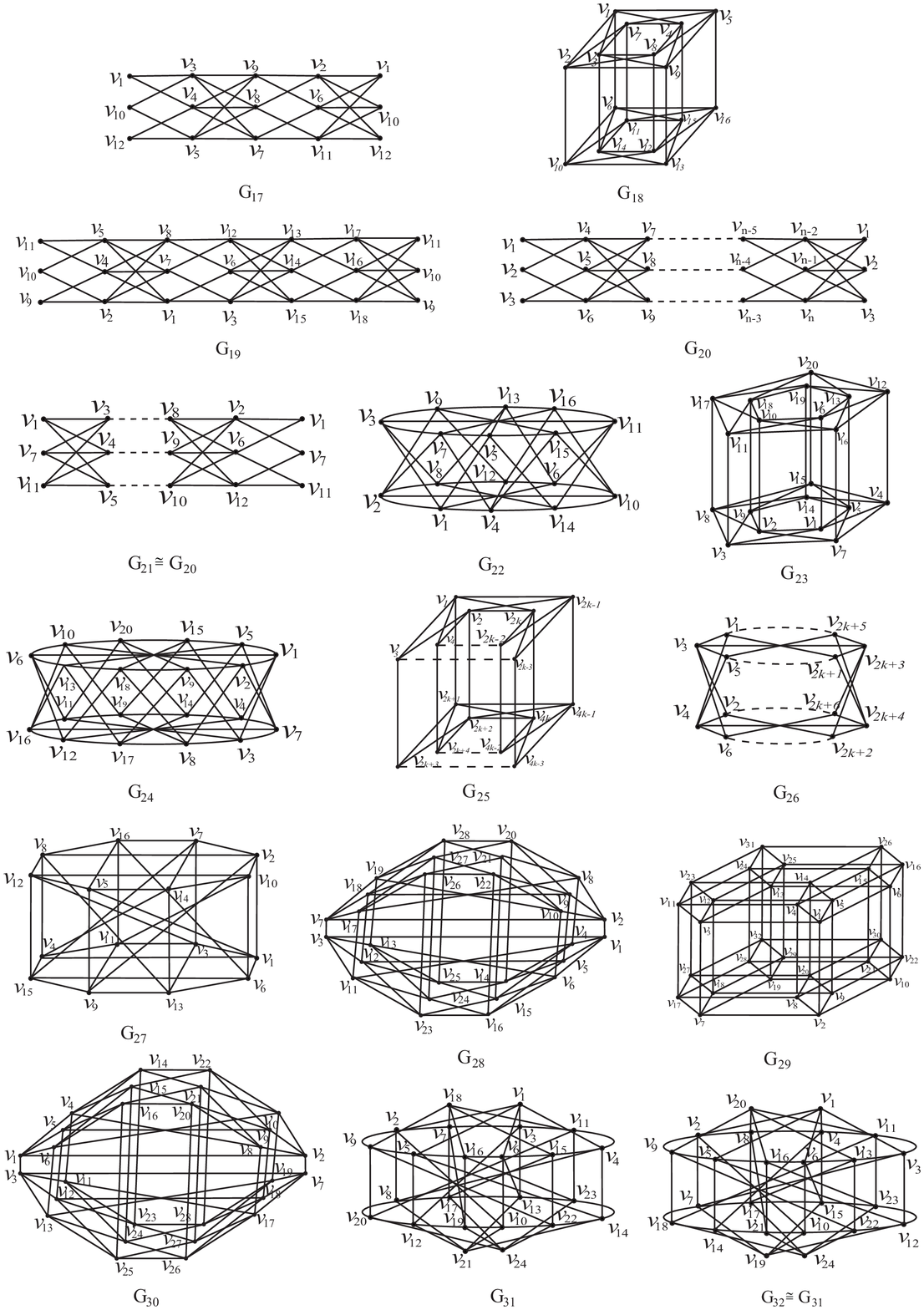}\\
\caption{All the possible underlying graphs for Theorem \ref{th:c-3} }\label{Fig-3}
\end{centering}
\end{figure}

%






\begin{proof}\noindent
   Since the graph $G$ is triangle-free, any two vertices which are
adjacent to a common vertex in $G$ are not be adjacent.


  Let $v_1\in V(G)$ and $N(v_1)=\{v_2,v_3,v_4,v_5,v_6\}$,
then $v_2$ is not adjacent to any vertices of $v_3,v_4,v_5,v_6$.
So there are four new neighbors $v_7,v_8,v_9,v_{10}$ of $v_2$.
Now $N(v_1)\cap N(v_7)\supseteq\{v_2\}$, by Lemma \ref{le:c-4}
we know the possible values of $|N(v_1)\cap N(v_7)|$ are 2 or
4. Similarly $|N(v_1)\cap N(v_8)|=$ 2 or 4, $|N(v_1)\cap N(v_9)|=$
2 or 4, $|N(v_1)\cap N(v_{10})|=$ 2 or 4.
Let $a_i=|N(v_1)\cap N(v_{i+6})|$, for $i=1,2,3,4$. By symmetry,
we only need to consider the following cases according to the vector $(a_1,a_2,a_3,a_4)$.
They are $(4,4,4,4)$, $(4,4,4,2)$, $(4,4,2,2)$, $(4,2,2,2)$, $(2,2,2,2)$.
We assume $S_1=\{v_3,v_4,v_5,v_6\}$ and $S_2=\{v_7,v_8,v_9,v_{10}\}$ for convenience.

\textit{Case 1.} $(a_1,a_2,a_3,a_4)=(4,4,4,4)$.

   Now every vertex of $S_1$ and $S_2$ has just three neighbors in the other set.

   Firstly, we prove the four sets
$N(v_7)\cap S_1, N(v_8)\cap S_1, N(v_9)\cap S_1, N(v_{10})\cap S_1$
are different each other. Otherwise, without loss of generality, let
$N(v_7)\cap S_1 = N(v_8)\cap S_1$, and $v_3v_7\in E(G)$, $v_4v_7\in
E(G)$, $v_5v_7\in E(G)$, $v_3v_8\in E(G)$, $v_4v_8\in E(G)$, and
$v_5v_8\in E(G)$. Now $N(v_7)\supseteq\{v_2,v_3,v_4,v_5\}$,
$N(v_8)\supseteq\{v_2,v_3,v_4,v_5\}$,
$N(v_1)=\{v_2,v_3,v_4,v_5,v_6\}$, obviously $v_6v_7\notin E(G)$, and
$v_6v_8\notin E(G)$. Since $N(v_2)\cap N(v_6)\supseteq\{v_1\}$,
$v_6v_9\in E(G)$ or $v_6v_{10}\in E(G)$,
and not both. While $N(v_9)\supseteq\{v_2\}$,
$N(v_{10})\supseteq\{v_2\}$, without loss of generality we assume
$v_6v_9\in E(G)$ and $v_6v_{10}\notin E(G)$.

  Since $a_3=|N(v_1)\cap N(v_9)|=4$ and $v_2v_9\in E(G)$, $v_6v_9\in E(G)$,
and there are two neighbors of $v_9$ in $\{v_3,v_4,v_5\}$.
Without loss of generality let
$v_3v_9\in E(G)$, $v_4v_9\in E(G)$ and $v_5v_9\notin E(G)$. Now
$N(v_3)\supseteq\{v_1,v_7,v_8,v_9\}$, $N(v_4)\supseteq\{v_1,v_7,v_8,v_9\}$,
$N(v_2)=\{v_1,v_7,v_8,v_9,v_{10}\}$. Then $v_3v_{10}\notin E(G)$,
$v_4v_{10}\notin E(G)$, and $a_4\leqslant5-2=3$,
this contradicts to $a_4=4$.

  Because $(a_1,a_2,a_3,a_4)=(4,4,4,4)$, $S_1=\{v_3,v_4,v_5,v_6\}$ and $S_2=\{v_7,v_8,v_9,v_{10}\}$,
every vertex of $S_2$ is not adjacent to one vertex of $S_1$.
From the proof above, we know any
two vertices of $S_2$ have not three common neighbors in $S_1$.
Without loss of generality, let $v_4v_7\in E(G)$, $v_5v_7\in E(G)$,
$v_6v_7\in E(G)$, $v_3v_7\notin E(G)$, $v_3v_8\in E(G)$, $v_5v_8\in
E(G)$, $v_6v_8\in E(G)$, $v_4v_8\notin E(G)$, $v_3v_9\in E(G)$,
$v_4v_9\in E(G)$, $v_6v_9\in E(G)$, $v_5v_9\notin E(G)$,
$v_3v_{10}\in E(G)$, $v_4v_{10}\in E(G)$, $v_5v_{10}\in E(G)$,
$v_6v_{10}\notin E(G)$. The four sets $N(v_7)\cap
S_1, N(v_8)\cap S_1, N(v_9)\cap S_1, N(v_{10})\cap S_1$ are different
to each other.

  Since the graph contains no $K_3$,
then $v_7v_8\notin E(G)$, $v_7v_9\notin E(G)$, $v_7v_{10}\notin
E(G)$. $v_3v_7\notin E(G)$, otherwise $|N(v_2)\cap N(v_3)|=5$.
Therefore there is a new neighbor $v_{11}$ of $v_7$.
Then $v_8v_{11}\in E(G)$ ($|N(v_7)\cap N(v_8)|$),
$v_9v_{11}\in E(G)$ ($|N(v_7)\cap N(v_9)|$),
$v_{10}v_{11}\in E(G)$ ($|N(v_7)\cap N(v_{10})|$).
Since the graph contains no triangle, similarly there is a
new neighbor $v_{12}$ of $v_3$.
Then $v_4v_{12}\in E(G)$ ($|N(v_3)\cap N(v_4)|$),
$v_5v_{12}\in E(G)$ ($|N(v_3)\cap N(v_5)|$),
$v_6v_{12}\in E(G)$ ($|N(v_3)\cap N(v_6)|$).
$|N(v_6)\cap N(v_{11})|$ and $d(v_1)=5$ imply $v_{11}v_{12}\in E(G)$.
Since $d(v_i)=5$ for $i=1,2,\ldots,12$, the graph is $G_{17}$ as
shown in Figure \ref{Fig-3}.

\textit{Case 2.} $(a_1,a_2,a_3,a_4)=(4,4,4,2)$.

  By using a method similar to that used in Case 1 of Theorem \ref{th:c-3},
we find the underlying graph $G_{18}$ depicted in Figure \ref{Fig-3}.

\textit{Case 3.} $(a_1,a_2,a_3,a_4)=(4,4,2,2)$.

 Similarly, we find $G_{19}$ and $G_{20}$ as shown in Figure \ref{Fig-3}.

\textit{Case 4.} $(a_1,a_2,a_3,a_4)=(4,2,2,2)$.

 Similarly, we find $G_{21},G_{22},G_{23},G_{24},G_{25},G_{26}$ as shown in Figure \ref{Fig-3}.

\textit{Case 5.} $(a_1,a_2,a_3,a_4)=(2,2,2,2)$.

  Now let $G^3=G[S_1\cup S_2]=G[{v_3,\ldots,v_{10}}]$.
Since $(a_1,a_2,a_3,a_4)=(2,2,2,2)$, $|E(G^3)|=(2-1)+(2-1)+(2-1)+(2-1)=4$.
$N(v_2)\cap N(v_3)=\{v_1\}$, so we claim every vertex of $S_1$ and $S_2$
has just one neighbor in the other set.
We say that such vertex sets $S_1$ and $S_2$ are {\it one-to-one} for convenience.

  From $(a_1,a_2,a_3,a_4)=(2,2,2,2)$, we know $N(v_1)-\{v_2\}=S_1$ and
$N(v_2)-\{v_1\}=S_2$ are {\it one-to-one}.
For any two adjacent vertices $u_1,u_2$ in $V(G)$, we can assume
that $N(u_1)-\{u_2\}$ and $N(u_2)-\{u_1\}$ are {\it one-to-one}, since
Case 5 is the last case of the discussion for proving Theorem \ref{th:c-3}.
If there are two adjacent vertices $u'_1,u'_2$ and $N(u'_1)-\{u'_2\}$ and
$N(u'_2)-\{u'_1\}$ are not {\it one-to-one}, then $N(u'_1)-\{u'_2\}$ and
$N(u'_2)-\{u'_1\}$ must belong to one of Case 1, 2, 3 or 4 by changing vertices
$u_1,u_2$ to $v_1,v_2$, respectively. So our assume is reasonable.

  By symmetry and the one-oneness of $S_1$ and $S_2$, we may assume
$v_3v_7\in E(G)$, $v_4v_8\in E(G)$, $v_5v_9\in E(G)$, and $v_6v_{10}\in E(G)$.
Then $v_3v_8\notin E(G)$, $v_3v_9\notin E(G)$,
$v_3v_{10}\notin E(G)$,  $v_4v_7\notin E(G)$, $v_4v_9\notin E(G)$,
$v_4v_{10}\notin E(G)$, $v_5v_7\notin E(G)$, $v_5v_8\notin E(G)$,
$v_5v_{10}\notin E(G)$, $v_6v_7\notin E(G)$, $v_6v_8\notin E(G)$, and
$v_6v_9\notin E(G)$. The graph is triangle-free, so $v_3v_4\notin
E(G)$, $v_3v_5\notin E(G)$, and $v_3v_6\notin E(G)$. Thus there are
three new neighbors of $v_3$, denoted by $v_{11},v_{12},v_{13}$.
Since $v_1v_3\in E(G)$, $N(v_1)-\{v_3\}=\{v_2,v_4,v_5,v_6\}$ and
$N(v_3)-\{v_1\}=\{v_7,v_{11},v_{12},v_{13}\}$ are {\it one-to-one}, coupled
with $v_2v_7\in E(G)$ and $N(v_{11})\supseteq\{v_3\}$,
$N(v_{12})\supseteq\{v_3\}$, and $N(v_{13})\supseteq\{v_3\}$.
Without loss of generality, let $v_4v_{11}\in E(G)$, $v_5v_{12}\in E(G)$,
and $v_6v_{13}\in E(G)$.

  The graph is triangle-free, so $v_3v_5\notin E(G)$, and $v_3v_6\notin E(G)$.
Since $N(v_1)-\{v_2\}$ and $N(v_2)-\{v_1\}$ are {\it one-to-one},
$v_4v_7\notin E(G)$, $v_4v_9\notin E(G)$, $v_4v_{10}\notin E(G)$.
Since $N(v_1)-\{v_3\}$ and $N(v_3)-\{v_1\}$ are {\it one-to-one},
$v_4v_{12}\notin E(G)$, $v_4v_{13}\notin E(G)$. There are two new
neighbors $v_{14},v_{15}$ of $v_4$. Since $v_1v_4\in E(G)$,
$N(v_1)-\{v_4\}$ and $N(v_4)-\{v_1\}$ are {\it one-to-one}.
Then $N(v_1)-\{v_4\}=\{v_2,v_3,v_5,v_6\},
N(v_4)-\{v_1\}=\{v_8,v_{11},v_{14},v_{15}\}$, $v_2v_8\in E(G)$,
$v_3v_{11}\in E(G)$, $N(v_{14})\supseteq\{v_4\}$, and
$N(v_{15})\supseteq\{v_4\}$. Without loss of generality, let
$v_5v_{14}\in E(G)$, and $v_6v_{15}\in E(G)$.

   The graph is triangle-free, so $v_5v_6\notin E(G)$,
$N(v_1)-\{v_2\}$ and $N(v_2)-\{v_1\}$ are {\it one-to-one}. So
$v_5v_7\notin E(G)$, $v_5v_8\notin E(G)$, $v_5v_{10}\notin E(G)$,
and $v_6v_9\notin E(G)$. Since $N(v_1)-\{v_3\}$ and $N(v_3)-\{v_1\}$
are {\it one-to-one}, so $v_5v_{11}\notin E(G)$,
$v_5v_{13}\notin E(G)$, $v_6v_{12}\notin E(G)$.
Since $N(v_1)-\{v_4\}$ and $N(v_4)-\{v_1\}$ are {\it one-to-one},
so $v_5v_{15}\notin E(G)$, $v_6v_{14}\notin E(G)$.
Then there is a new neighbor of $v_5$, say $v_{16}$.
$|N(v_5)\cap N(v_6)|$ implies $v_6v_{16}\in E(G)$ or $v_6v_{12}\in E(G)$.
Since $N(v_1)-\{v_3\}$ and $N(v_3)-\{v_1\}$ are {\it one-to-one},
$v_6v_{12}\notin E(G)$. Thus $v_6v_{16}\in E(G)$.
The graph is triangle-free so $v_7v_8\notin E(G)$, $v_7v_9\notin E(G)$,
$v_7v_{10}\notin E(G)$, $v_7v_{11}\notin E(G)$, $v_7v_{12}\notin E(G)$,
and $v_7v_{13}\notin E(G)$.
Let $B=\{v_7v_{14},v_7v_{15},v_7v_{16},v_8v_{12},v_8v_{13},v_8v_{16},
v_9v_{11},$ $v_9v_{13},v_9v_{15},v_{10}v_{11},v_{10}v_{12},v_{10}v_{14},
v_{11}v_{16},v_{12}v_{15},$\\$v_{13}v_{14}\}$.
Next we will discuss the following two cases.

\textit{Case 5.1.} $E(G)\cap B\neq\emptyset$.

 By using a method similar to that used in Case 1 of Theorem \ref{th:c-3},
we find $G_{27}$ as shown in Figure \ref{Fig-3}.

\textit{Case 5.2.} $E(G)\cap B=\emptyset$.

  Similarly, we find $G_{28},G_{29},G_{30},G_{31},G_{32}$ as shown in Figure \ref{Fig-3}.

  Now we have found all possible underlying graphs of
connected 5-regular optimum skew energy oriented graphs with the even neighborhood property.
\end{proof}

By Theorems \ref{th:c-1}, \ref{th:c-2} and \ref{th:c-3}, we obtain fourteen classes of all possible
underlying graphs of connected 5-regular oriented graphs of order $n$ with optimum skew energy.
These fourteen classes of underlying graphs are
(\uppercase\expandafter{\romannumeral 1}) $G_1$,
(\uppercase\expandafter{\romannumeral 2}) $G_2$, $G_3$, $G_{9}$, $G_{10}$,
(\uppercase\expandafter{\romannumeral 3}) $G_4$, $G_5$, $G_6$,
(\uppercase\expandafter{\romannumeral 4}) $G_7$, $G_8$, $G_{11}$, $G_{12}$, $G_{13}$,
(\uppercase\expandafter{\romannumeral 5}) $G_{14}$,
(\uppercase\expandafter{\romannumeral 6}) $G_{15}$,
(\uppercase\expandafter{\romannumeral 7}) $G_{16}$,
(\uppercase\expandafter{\romannumeral 8}) $G_{17},G_{19},G_{20},G_{21}$,
(\uppercase\expandafter{\romannumeral 9}) $G_{18},G_{23},G_{25}$,
(\uppercase\expandafter{\romannumeral 10}) $G_{22},G_{24},G_{26}$,
(\uppercase\expandafter{\romannumeral 11}) $G_{27}$,
(\uppercase\expandafter{\romannumeral 12}) $G_{28}, G_{30}$,
(\uppercase\expandafter{\romannumeral 13}) $G_{29}$,
(\uppercase\expandafter{\romannumeral 14}) $G_{31},G_{32}$.

\section{All connected 5-regular optimum skew energy oriented graphs}
\label{oriented-sec}

  In this section, we will prove exactly eleven classes
of connected 5-regular graphs obtained in Section 2 have
optimum orientations.

  Let $C_n=v_1v_2,\ldots,v_n$ be a cycle on $n(\geq3)$ vertices and
$C'_{n}=u_1u_2,\ldots,u_n$ is a copy
of $C$. We define $U_n=(V(U_n),E(U_n))$ to be the graph obtained
from $C_n$ and $C'_{n}$ by inserting edges from $v_i$ to $u_{i-1(mod
\ n)}$ and $u_{i+1(mod \ n)}$ for $i=1,2,\ldots,n$, see Fig. 2.1 in \cite{GXZ}.

  The Cartesian product of $G_1$ and $G_2$, denoted by $G_1\square G_2$,
is defined as
a graph with vertex set $V(G_{1})\times V(G_{2})$.
Suppose $v_1, v_2\in V(G_1)$ and $u_1,u_2\in V(G_2)$.
Then $(u_1,v_1)$ and $(u_2,v_2)$ are adjacent in $G_1\square G_2$ if and only if
$u_1u_2\in E(G_1),v_1=v_2$ or $u_1=u_2,v_1v_2\in E(G_2)$.
For example, let $\mathbf{Q_1}=P_2$, and $\mathbf{Q_n}=\mathbf{Q_{n-1}}\square P_2$
for $n\geq2$. The graph $\mathbf{Q_n}$ is called the hypercube graph.




\begin{lemma}(Lemma 2.8 of \cite{ABW}) \label{le:o-1} If
$K_n$ has an orientation with $S(G^\sigma)^{T}S(G^\sigma)=(n-1)I_n$,
then n is a multiple of 4.
\end{lemma}

\begin{lemma}(Theorem 3.5 of \cite{CH}) \label{le:o-2} Let
$G^\sigma$ be an oriented k-regular graph of $G$ on n vertices with
maximum skew energy $\varepsilon_s(G^\sigma)=n\sqrt k$. Then the
oriented graph $(P_{2}\Box G)^o$ of $P_{2}\Box G$ has the maximum
skew energy $\varepsilon_s(G^\sigma)=2n\sqrt{k + 1}$.
\end{lemma}

  Let $S(G^\sigma)=[s_{ij}]_{n\times n}$ be the skew adjacency matrix of $G^\sigma$
and let $W=u_1u_2 \ldots u_k$ (perhaps $u_i=u_j$ for $i\neq j$) be a walk in $G^\sigma$
joining $u_1$ and $u_k$. The sign of $W$, denoted by $sgn(W)$, is
defined as
$$sgn(W)=\prod_{i=1}^{k-1}{s_{u_iu_{i+1}}}.$$
Denote by
$w_{uv}^{+}(k)$ and $w_{uv}^{-}(k)$ the number of all positive walks
and negative walks starting $u$ ending $v$ with the length $k$,
respectively.

\begin{lemma}(Lemma 3.7 of \cite{GX}, Proposition 2.4 of \cite{CLL}) \label{le:o-3} Let
$G^\sigma$ be a $k$-regular oriented graph on $n$ vertices with
skew adjacency matrix $S=S(G^\sigma)$. then $S^TS=kI_n$ if and only
if for any two distinct vertices $u$ and $v$ of $G^\sigma$,
$$w_{uv}^{+}(2)= w_{uv}^{-}(2).$$
\end{lemma}

  In \cite{GX}, Gong and Xu presented the following method to construct
a new oriented graph with the optimum skew energy.

  Let $v$ be an arbitrary vertex of the oriented graph $G^\sigma$.
The operation by reversing the orientations of all arcs incident
with $v$ and preserving the orientations of all its other arcs is
called \emph{a reversal of $G^\sigma$ at $v$}. Denote by
$S(G^\sigma)$ and $S(G^\beta)$ the skew adjacency matrices of the
oriented graph $G^\sigma$ and $G^\beta$, respectively. Let $G^\beta$
be a reversal of $G^\sigma$ at $v$, then $S(G^\sigma)=
PS(G^\beta)P^{-1}$. Hence $$\varepsilon_s(G^\sigma)=\varepsilon_s(G^\beta),$$
where $P$ is the diagonal matrix obtained from the identity matrix $I_n$ by replacing
the diagonal entry corresponding to vertices $v$ by $-1$.

  We will determine all connected 5-regular optimum skew energy oriented graphs
as follow.

\begin{figure}[htbp]
\begin{centering}
\includegraphics[scale=0.6]{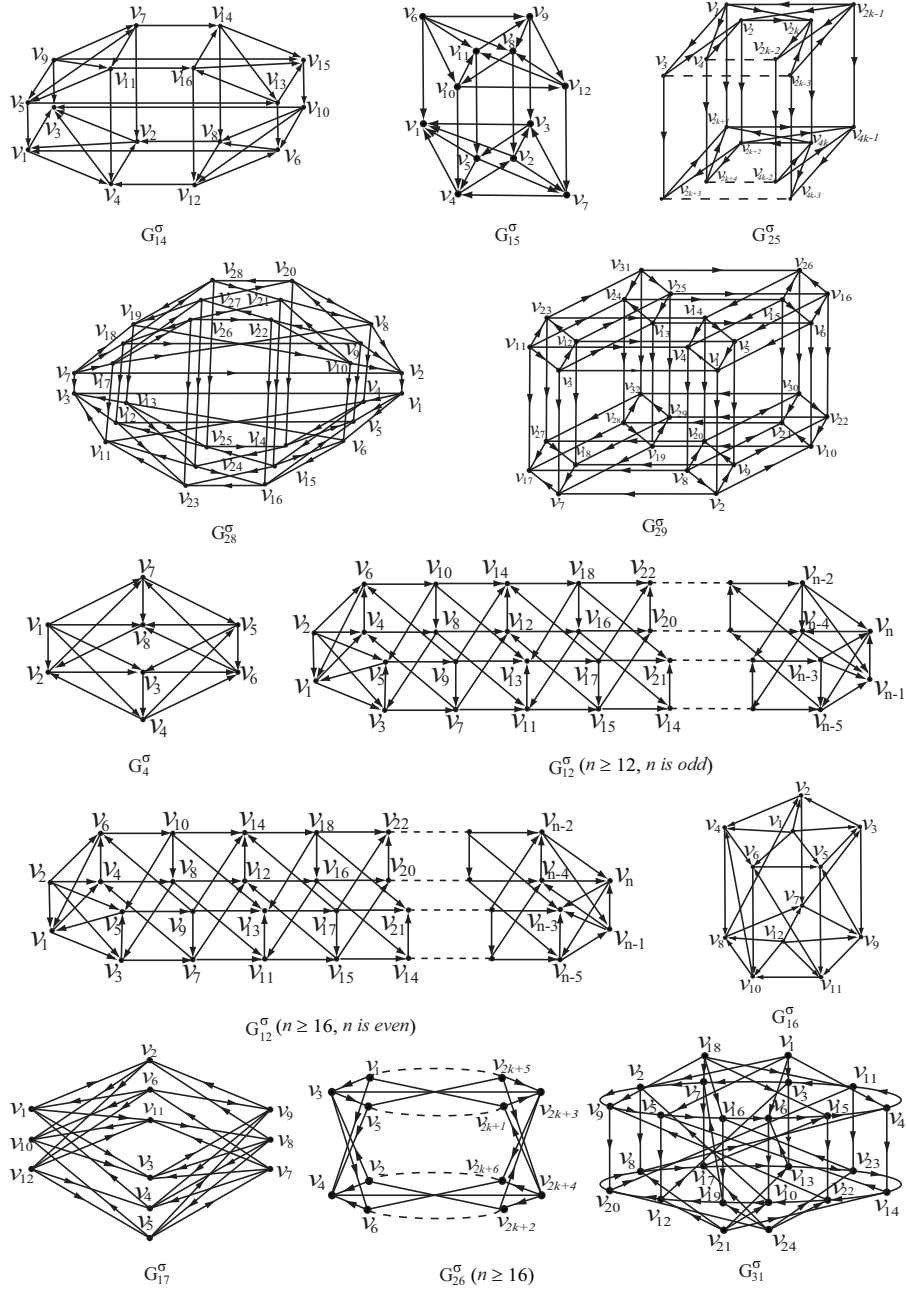}\\
\caption{Eleven classes of connected 5-regular optimum skew energy oriented
graphs.}\label{Fig-4}
\end{centering}
\end{figure}

\begin{theorem} \label{th:o-1} Let $G^\sigma$ be a connected 5-regular
optimum skew energy oriented
graph of order $n$. Then $G^\sigma$ is one of graphs with even neighborhood property
depicted in Figure \ref{Fig-4}.
\end{theorem}








\begin{proof}In the following, we will prove three classes of graphs
(\uppercase\expandafter{\romannumeral 1}),
(\uppercase\expandafter{\romannumeral 2}) and
(\uppercase\expandafter{\romannumeral 11}) (see the end of Section 2) have no optimum
orientations, and the other classes of graphs have optimum
orientations.

Each graph in the classes (\uppercase\expandafter{\romannumeral 5}),
(\uppercase\expandafter{\romannumeral 6}),
(\uppercase\expandafter{\romannumeral 9}),
(\uppercase\expandafter{\romannumeral 12}), and
(\uppercase\expandafter{\romannumeral 13}) has an optimum orientation.

For $G_{14}$, $G_{15}$, $G_{18}$, $G_{23}$, $G_{25}$, $G_{28}$ and $G_{29}$,
from \cite{CLL} and \cite{GXZ}, we know $U_n,(n\geq3)$, $P_2\Box K_4$, $Q_4$ and $G_3^\sigma$ (Fig. 3.4 in \cite{CLL}) have optimum orientations, respectively.
By Lemma \ref{le:o-2}, we obtain $G_{14}\cong P_2\Box(P_2\Box K_4)$, $G_{15}\cong P_2\Box U_3 $,
$G_{18}\cong P_2\Box U_4$, $G_{23}\cong P_2\Box U_5$, $G_{25}\cong P_2\Box U_n$,
$G_{28}\cong P_2\Box G_3$ ($G_3$ of Fig. 3.4. in \cite{CLL}), and
$G_{29}\cong P_2\Box Q_4$ have optimum orientations, respectively.

Only graph $G_4$ of class (\uppercase\expandafter{\romannumeral 3}) has an
optimum orientation.

For the graph $G_6$ with order $n\geq8$. we will prove that $G_6$ has an optimum
orientation only when $n=8$.

  Let \{$(v_1, v_2), (v_1, v_{n-1}), (v_1, v_n), (v_1,
v_3), (v_1, v_4)\}\subset E(G_6^\sigma)$. (We give a reversal of $G_6^\sigma$ at
$v_2, v_3, v_4, v_{n-1}$ or $v_n$ if necessary).

   Since $|V|=n\geq8$, the orientations of edges which are adjacent to $v_6$
are not determined, we can assume that $(v_4, v_6)\in E(G_6^\sigma)$,
$(v_3, v_5)\in E(G_6^\sigma)$.
(We give a reversal at $v_6$ or $v_5$ if necessary.) By Lemma \ref{le:o-3},
$w_{v_1v_6}^{+}(2)= w_{v_1v_6}^{-}(2)$ and
$w_{v_1v_5}^{+}(2)= w_{v_1v_5}^{-}(2)$, thus $(v_6, v_3)\in E(G_6^\sigma)$
and $(v_5, v_4)\in E(G_6^\sigma)$, respectively. Next we will discuss the
following two cases.

\textit{Case 1.} $(v_2, v_4)\in E(G_6^\sigma)$.

  By Lemma \ref{le:o-3}, we have $w_{v_1v_4}^{+}(2)= w_{v_1v_4}^{-}(2)$, thus
$(v_4, v_3)\in E(G_6^\sigma)$. Similarly,
$w_{v_1v_3}^{+}(2)= w_{v_1v_3}^{-}(2)$, thus $(v_3, v_2)\in E(G_6^\sigma)$.
Then we find that $w_{v_2v_6}^{+}(2)=2$, so there must be two
negative paths adjoining $v_2$ and $v_6$. This means $v_6$ must be
adjacent to $v_{n-1}$ and $v_n$. So there must be eight vertices in
$G_6^\sigma$.

\textit{Case 2.} $(v_4, v_2)\in E(G_6^\sigma)$.

By Lemma \ref{le:o-3}, we have $w_{v_1v_4}^{+}(2)=w_{v_1v_4}^{-}(2)$, thus
$(v_3, v_4)\in E(G_6^\sigma)$. Similarly,
$w_{v_1v_3}^{+}(2)= w_{v_1v_3}^{-}(2)$, thus $(v_2, v_3)\in E(G_6^\sigma)$.
Then we find that $w_{v_2v_6}^{-}(2)=2$, so there must be two
positive paths adjoining $v_2$ and $v_6$. This means $v_6$ must be
adjacent to $v_{n-1}$ and $v_n$. Hence there are eight vertices
in the graph, too.

   Therefore there are eight vertices in $G_6^\sigma$.
In fact, it is the graph which is isomorphic to $G_4^\sigma$.

   The skew adjacency matrix of $G_4^\sigma$ is the following

\begin{equation*}
S(G_4^\sigma)=
\left[\begin{array}{rrrrrrrr}
0& 1& 1& 1& 0& 0& 1& 1\\
-1& 0& 1& -1& 0& 0& 1& -1\\
-1& -1& 0& 1& -1& 1& 0& 0\\
-1& 1& -1& 0& 1& 1& 0& 0\\
0& 0& 1& -1& 0& 1& -1& 1\\
0& 0& -1& -1& -1& 0& 1& 1\\
-1& -1& 0& 0& 1& -1& 0& 1\\
-1& 1& 0& 0& -1& -1& -1& 0
\end{array}\right].
\end{equation*}

 It is easy to verify $S(G_4^\sigma)^TS(G_4^\sigma)=5I_n$. Then only graph $G_4$ of
class (\uppercase\expandafter{\romannumeral 3}) has an optimum orientation.

 We now show that each graph of class (\uppercase\expandafter{\romannumeral 4})
has an optimum orientation.

 For $G_{12}$ of order $n\geq 12$. We know $n$ must be multiple of 4.
While the parity of $n/4$ is not determined, we will give two kinds of
skew adjacency matrices of $G_{12}^\sigma$ according to the parity of
$n/4$. Here let
\begin{equation*}
D_1= \left[\begin{array}{rr}
0& -1\\
1& 0
\end{array}\right],
D_2= \left[\begin{array}{rrrr}
0& 0& 1& 0\\
0& 0& 0& 1\\
-1& 0& 0& 0\\
0& -1& 0& 0
\end{array}\right],
Q_1= \left[\begin{array}{rrrr}
1& 1& -1& -1\\
1& 1& 1& 1
\end{array}\right],
\end{equation*}
\begin{equation*}
Q_2= \left[\begin{array}{rrrr}
1& -1& 0& 0\\
-1& 1& 0& 0\\
0& 0& 1& -1\\
0& 0& -1& 1
\end{array}\right],
Q_3= \left[\begin{array}{rrrr}
1& 1& 0& 0\\
1& 1& 0& 0\\
0& 0& 1& 1\\
0& 0& 1& 1
\end{array}\right],
Q_4= \left[\begin{array}{rr}
1& -1\\
-1& 1\\
-1& -1\\
1& 1
\end{array}\right],
Q_5= \left[\begin{array}{rr}
1& -1\\
1& -1\\
1& 1\\
1& 1
\end{array}\right].
\end{equation*}

   (1) $n/4$ is odd and $n\geq12$.

Let the rows of $S(G_{12}^\sigma)$ correspond successively the vertices
$v_1$, $v_2$, \ldots, $v_n$.

\begin{equation*}
S(G_{12}^\sigma)\begin{tiny}= \left[\begin{matrix}
D_1& Q_1\\
-Q_1^T& D_2& Q_2\\
& -Q_2^T& -D_2& Q_3\\
& & -Q_3^T& D_2& Q_2\\
& & & -Q_2^T& -D_2& Q_3\\
& & & & \ddots& \ddots& \ddots\\
& & & & & -Q_3^T& D_2& Q_2\\
& & & & & & -Q_2^T& -D_2& Q_5\\
& & & & & & & -Q_5^T& D_1^T\\
\end{matrix}\right].
\end{tiny}
\end{equation*}

   Next, we will prove $S(G_{12}^\sigma)^TS(G_{12}^\sigma)=5I_n$.

\begin{equation*}
\begin{aligned}
-S[G_{12}^\sigma]^TS[G_{12}^\sigma]=-S[G_{12}^\sigma]^2=
\begin{tiny}\left[\begin{matrix}
A& G& K\\
P& B& G& L\\
K^T& Q& C& I& M\\
& L^T& R& D& H& L\\
& & M^T& Q& C& I& M\\
& & \ddots& \ddots& \ddots& \ddots& \ddots\\
& & & L^T& R& D& G& O\\
& & & & M^T& Q& E& J\\
& & & & & O^T& U& F
\end{matrix}\right],
\end{tiny}
\end{aligned}
\end{equation*}
where $A={D_1}^2-Q_1{Q_1}^T$,
$B=-Q_1^TQ_1+{D_2}^2-Q_2{Q_2}^T$
$C=-Q_2^TQ_2+{D_2}^2-Q_3{Q_3}^T$,
$D=-Q_3^TQ_3+{D_2}^2-Q_2{Q_2}^T$,
$E=-Q_2^TQ_2+{D_2}^2-Q_5Q_5^T$,
$F=-Q_5^TQ_5+{D_1^T}^2$,
$G=D_1Q_1+Q_1D_2$,
$H=D_2Q_2-Q_2D_2$,
$I=-D_2Q_3+Q_3D_2$,
$J=-D_2Q_5+Q_5{D_1}^T$,
$K=Q_1Q_2$,
$L=Q_2Q_3$,
$M=Q_3Q_2$,
$O=Q_2Q_5$,
$P=-Q_1^TD_1-D_2Q_1^T$,
$Q=-Q_2^TD_2+D_2Q_2^T$,
$R=Q_3^TD_2-D_2Q_3^T$,
$U=Q_5^TD_2-D_1^TQ_5^T$.

   We obtain $A=D_1^2-Q_1Q_1^T=-5I_2$, $B=-Q_1^TQ_1+{D_2}^2-Q_2Q_2^T=-5I_4$,
$C=-Q_2^TQ_2+{D_2}^2-Q_3{Q_3}^T=-5I_4$,
$D=-Q_3^TQ_3+{D_2}^2-Q_2Q_2^T=-5I_4$,
$E=-Q_2^TQ_2+{D_2}^2-Q_5Q_5^T=-5I_4$,
$F=-Q_5^TQ_5+{D_1^T}^2=-5I_2$, $G=D_1Q_1+Q_1D_2=0$, $H=D_2Q_2-Q_2D_2=0$,
$I=-D_2Q_3+Q_3D_2=0$, $J=-D_2Q_5+Q_5D_1^T=0$,
$K=Q_1Q_2=0$, $L=Q_2Q_3=0$, $M=Q_3Q_2=0$, $O=Q_2Q_5=0$.

Since $S(G_{12}^\sigma)$ is an antisymmetric matrix, $-(S(G_{12}^\sigma))^2$ is a
symmetric matrix, then
$$S(G_{12}^\sigma)^{T}S(G_{12}^\sigma)=-(S(G_{12}^\sigma))^2=5I_n.$$

  (2) $n/4$ is even and $n\geq16$.

  Let the skew adjacency matrix of graph $G_{12}^\beta$ be $S(G_{12}^\beta)$,
and the rows of $S(G_{12}^\beta)$ correspond successively the
vertices $v_1$, $v_2$, \ldots, $v_n$.

\begin{tiny}
\begin{equation*}
S(G_{12}^\beta)= \left[\begin{matrix}
D_1& Q_1\\
-Q_1^T& D_2& Q_2\\
& -Q_2^T& -D_2& Q_3\\
& & -Q_3^T& D_2& Q_2\\
& & & -Q_2^T& -D_2& Q_3\\
& & & & \ddots& \ddots& \ddots\\
& & & & & -Q_2^T& -D_2& Q_3\\
& & & & & & -Q_3^T& D_2& Q_4\\
& & & & & & & -Q_4^T& D_1^T\\
\end{matrix}\right].
\end{equation*}
\end{tiny}

  Similarly we can prove $S(G_{12}^\beta)^TS(G_{12}^\beta)= 5I_n$.

By Lemma \ref{le:c-1}, we know that each graph in class
(\uppercase\expandafter{\romannumeral 4}) has an optimum orientation.

We will prove each graph in class
(\uppercase\expandafter{\romannumeral 7}) has an optimum orientation.

For $G_{16}$, let the rows of $S(G_{16}^\sigma )$ correspond
successively the vertices $v_3,v_2,v_4,v_6,v_5,v_1,v_9,$ $v_7,v_8,v_{10},v_{11},v_{12}$. It follows that,
\begin{equation*}
S(G_{16}^\sigma)=\left[\begin{array}{rrrrrrrrrrrr}
0& 1& 0& 0& -1& -1& 1& -1& 0& 0& 0& 0\\
-1& 0& 1& 0& 0& -1& 0& 1& -1& 0& 0& 0\\
0& -1& 0& 1& 0& -1& 0& 0& 1& -1& 0& 0\\
0& 0& -1& 0& 1& -1& 0& 0& 0& 1& -1& 0\\
1& 0& 0& -1& 0& -1& -1& 0& 0& 0& 1& 0\\
1& 1& 1& 1& 1& 0& 0& 0& 0& 0& 0& 0\\
-1& 0& 0& 0& 1& 0& 0& -1& 0& 0& 1& -1\\
1& -1& 0& 0& 0& 0& 1& 0& -1& 0& 0& -1\\
0& 1& -1& 0& 0& 0& 0& 1& 0& -1& 0& -1\\
0& 0& 1& -1& 0& 0& 0& 0& 1& 0& -1& -1\\
0& 0& 0& 1& -1& 0& -1& 0& 0& 1& 0& -1\\
0& 0& 0& 0& 0& 0& 1& 1& 1& 1& 1& 0
\end{array}\right].
\end{equation*}

 It is easy to check that $S(G_{16}^\sigma)^TS(G_{16}^\sigma)=5I_{12}$.
Then each graph in class (\uppercase\expandafter{\romannumeral 7})
has an optimum orientation.

For the graphs of class (\uppercase\expandafter{\romannumeral 8}),
such as $G_{20}$ of order $n\geq12$, we will prove that only when
$n=12$ (it is $G_{17}$), $G_{20}$ has an optimum orientation.
Next we give the skew adjacency matrix of $G_{17}^\sigma$ with rows corresponding to
$v_1,v_{10},v_{12},v_3,v_4,v_5,v_9,v_8,v_7,v_2,v_6,v_{11}$.

\begin{equation*}
S(G_{17}^\sigma)= \left[\begin{array}{rrrrrrrrrrrr}
0& 0& 0& 1& -1& 0& 0& 0& 0& 1& -1& 1\\
0& 0& 0& 1& 0& -1& 0& 0& 0& -1& -1& -1\\
0& 0& 0& 0& 1& 1& 0& 0& 0& -1& -1& 1\\
-1& -1& 0& 0& 0& 0& 1& -1& -1& 0& 0& 0\\
1& 0& -1& 0& 0& 0& -1& -1& -1& 0& 0& 0\\
0& 1& -1& 0& 0& 0& 1& -1& 1& 0& 0& 0\\
0& 0& 0& -1& 1& -1& 0& 0& 0& 1& -1& 0\\
0& 0& 0& 1& 1& 1& 0& 0& 0& 1& 0& -1\\
0& 0& 0& 1& 1& -1& 0& 0& 0& 0& 1& 1\\
-1& 1& 1& 0& 0& 0& -1& -1& 0& 0& 0& 0\\
1& 1& 1& 0& 0& 0& 1& 0& -1& 0& 0& 0\\
-1& 1& -1& 0& 0& 0& 0& 1& -1& 0& 0& 0
\end{array}\right].
\end{equation*}

   It is easy to verify that $S(G_{17}^\sigma)^TS(G_{17}^\sigma)=5I_{12}$.

   For the graph $G_{20}$ with order $n >12$, from the structure of $G_{20}$
we know $n \geq18$.
Let $\{(v_1, v_4),(v_1,v_5),(v_1, v_{n-2}),(v_{n-1}, v_1),(v_n, v_1)\}\subset E(G^\sigma_{20})$,
otherwise we can give a reversal of $G^\sigma_{20}$ at $v_4, v_5, v_{n-2},v_{n-1}, v_n$.

   Let $(v_4, v_7)\in E(G_{20}^\sigma)$ (we give a reversal of $G_{20}^\sigma$
at $v_7$ if necessary).
$v_1v_4v_7$ and $v_1v_5v_7$ are two paths between $v_1$ and $v_7$ with length 2.
$v_1v_4v_7$ is positive, we have $v_1v_5v_7$ is negative by applying Lemma \ref{le:o-3},
thus $(v_7, v_5)\in E(G_{20}^\sigma)$.

   Similarly, let $(v_4, v_8)\in E(G_{20}^\sigma)$, $(v_4, v_9)\in E(G_{20}^\sigma)$
(we give a reversal of $G_{20}^\sigma$ at $v_8$ or $v_9$ if necessary).
By applying Lemma \ref{le:o-3}, $w_{v_1v_8}^{+}(2)=w_{v_1v_8}^{-}(2)$,
$w_{v_1v_9}^{+}(2)=w_{v_1v_9}^{-}(2)$,
thus $(v_7, v_5)\in E(G_{20}^\sigma)$, $(v_8, v_5)\in E(G_{20}^\sigma)$,
$(v_9, v_5)\in E(G_{20}^\sigma)$.

   Let $(v_4, v_2)\in E(G_{20}^\sigma)$
(we give a reversal of $G_{20}^\sigma$ at $v_2$ if necessary).
Let $(v_6, v_2)\in E(G_{20}^\sigma)$
(we give a reversal of $G_{20}^\sigma$ at $v_6$ if necessary).
By applying Lemma \ref{le:o-3}, $w_{v_2v_7}^{+}(2)=w_{v_2v_7}^{-}(2)$,
$w_{v_2v_8}^{+}(2)=w_{v_2v_8}^{-}(2)$,
$w_{v_2v_9}^{+}(2)=w_{v_2v_9}^{-}(2)$,
thus $(v_7, v_6)\in E(G_{20}^\sigma)$, $(v_8, v_6)\in E(G_{20}^\sigma)$,
$(v_9, v_6)\in E(G_{20}^\sigma)$.

   Let $(v_5, v_3)\in E(G_{20}^\sigma)$ (we give a reversal of $G_{20}^\sigma$ at $v_3$ if necessary).
   Since $w_{v_3v_7}^{+}(2)= w_{v_3v_7}^{-}(2)$, $(v_3, v_6)\in E(G_{20}^\sigma)$.

   $w_{v_4v_{n-2}}^{+}(2)= w_{v_4v_{n-2}}^{-}(2)$, $w_{v_4v_{n-1}}^{+}(2)=w_{v_4v_{n-1}}^{-}(2)$,
$w_{v_4v_{n}}^{+}(2)=w_{v_4v_{n}}^{-}(2)$, thus $(v_2, v_{n-2})\in E(G_{20}^\sigma)$,
$(v_{n-1}, v_2)\in E(G_{20}^\sigma)$, $(v_n, v_2)\in E(G_{20}^\sigma)$.

   $w_{v_6v_{n-2}}^{+}(2)=w_{v_6v_{n-2}}^{-}(2)$, $w_{v_6v_{n-1}}^{+}(2)=w_{v_6v_{n-1}}^{-}(2)$,
$w_{v_6v_{n}}^{+}(2)=w_{v_6v_{n}}^{-}(2)$, thus $(v_3, v_{n-2})\in E(G_{20}^\sigma)$,
$(v_{n-1}, v_3)\in E(G_{20}^\sigma)$, $(v_n, v_3)\in E(G_{20}^\sigma)$.

   Since $n >12$, we know that any edge which is adjacent to $v_{10}$
is not oriented,
so let $(v_7, v_{10})\in E(G_{20}^\sigma)$ (we give a reversal of $G_{20}^\sigma$
at $v_{10}$ if necessary).
Similarly let $(v_7, v_{11})\in E(G_{20}^\sigma)$ (we give a reversal of
$G_{20}^\sigma$ at $v_{11}$ if necessary).
By applying Lemma \ref{le:o-3}, $w_{v_4v_{10}}^{+}(2)=w_{v_4v_{10}}^{-}(2)$,
$w_{v_4v_{11}}^{+}(2)=w_{v_4v_{11}}^{-}(2)$,
thus $(v_{10}, v_8)\in E(G_{20}^\sigma)$, $(v_{11}, v_9)\in E(G_{20}^\sigma)$.

   Let $(v_9, v_{12})\in E(G_{20}^\sigma)$ (we give a reversal of
$G_{20}^\sigma$ at $v_{12}$ if necessary).
By applying Lemma \ref{le:o-3}, $w_{v_6v_{12}}^{+}(2)=w_{v_6v_{12}}^{-}(2)$,
thus $(v_{12}, v_8)\in E(G_{20}^\sigma)$.

  We find $w_{v_7v_8}^{+}(2)=1$, $w_{v_7v_8}^{-}(2)=3$,
this is a contradiction to Lemma \ref{le:o-3}.
So each graph of class (\uppercase\expandafter{\romannumeral 8})
have no optimum orientations for $n >12$ .


Each graph of class (\uppercase\expandafter{\romannumeral 10}) has an optimum orientation.

  For $G_{26}^\sigma$ of order $n \geq 16$, let $n=4k$
for convenience and the rows of $S(G_{26}^\sigma)$
correspond successively the vertices $v_1,v_2,\ldots,v_n$.
\begin{equation*}
S(G_{26}^\sigma)= \left[\begin{array}{c;{2pt/2pt}c}
A_{1}& A_{2}\\
\hdashline[2pt/2pt]
-A_{2}^T& A_{3}\\
\end{array}\right].
\end{equation*}
  Because the skew adjacency matrix $S(G_{26}^\sigma)$ is skew symmetric,
so $A_1$ and $A_3$ are also skew symmetric matrices. Let

\begin{equation*} M=
\left[\begin{smallmatrix}
1& 0\\
0& 1\\
\end{smallmatrix}\right],
 M_1=
\left[\begin{smallmatrix}
-1& 0\\
0& 1\\
\end{smallmatrix}\right],
 M_2=
\left[\begin{smallmatrix}
1& 1\\
1& 1\\
\end{smallmatrix}\right],
M_3=
\left[\begin{smallmatrix}
1& 1\\
-1& -1\\
\end{smallmatrix}\right],
M_4=
\left[\begin{smallmatrix}
-1& 1\\
1& -1\\
\end{smallmatrix}\right].
\end{equation*}

\begin{tiny}
\begin{equation*}
A_1= \left[\begin{smallmatrix}
0& M_2\\
-M_2^T& 0& M_3\\
& -M_3^T& 0& \ddots\\
& &  \ddots& \ddots& M_3 \\
& & & -M_3^T& 0& M_3\\
& & & & -M_3^T& 0\\
\end{smallmatrix}\right],
A_2=
\left[\begin{smallmatrix}
 M& & & & -M_3\\
 & M_1\\
 & & M_1\\
 & & & \ddots\\
 M_4& & &  & M_1\\
\end{smallmatrix}\right],
A_3=\left[\begin{smallmatrix}
 0& M_3^T& \\
 -M_3& 0& -M_3^T\\
  & M_3& 0& \ddots\\
  & & \ddots& \ddots& -M_3^T\\
  & & & M_3& 0& -M_3^T\\
  & & & & M_3& 0\\
\end{smallmatrix}\right].
\end{equation*}
\end{tiny}

Then
\begin{equation*}
S(G_{26}^\sigma)^TS(G_{26}^\sigma)=-S(G_{26}^\sigma)^2= -\left[\begin{array}{cc}
 A_1^2-A_2A_2^T& A_1A_2+A_2A_3\\
 -A_2^TA_1-A_3A_2^T& -A_2^TA_2+A_3^2\\
\end{array}\right].
\end{equation*}

\begin{tiny}
\begin{equation*}
\begin{aligned}
&A_1A_2+A_2A_3\\
&=\left[\begin{smallmatrix}
0& M_2M_1\\
-M_2^TM& 0& M_3M_1& & & M_2^TM_3\\
& -M_3^TM_1& 0& \ddots\\
& &  \ddots& \ddots& M_3M_1 \\
M_3M_4& & & -M_3^TM_1& 0& M_3M_1\\
& & & & -M_3^TM_1& 0\\
\end{smallmatrix}\right]+
\left[\begin{smallmatrix}
  0& MM_3^T& & & -M_3^2\\
 -M_1M_3& 0& -M_1M_3^T\\
  & M_1M_3& 0& \ddots\\
  & & \ddots& \ddots& -M_1M_3^T\\
  & & & M_1M_3& 0& -M_1M_3^T\\
  & M_4M_3^T& & & M_1M_3& 0\\
\end{smallmatrix}\right]\\
&=\left[\begin{smallmatrix}
 0& M_2M_1+MM_3^T& & & -M_3^2& 0\\
 -M_2^TM-M_1M_3& 0& M_3M_1-M_1M_3^T& & & M_2^TM_3\\
  & -M_3^TM_1+M_1M_3& 0& \ddots\\
  & & \ddots& \ddots& M_3M_1-M_1M_3^T\\
  M_3M_4& & & -M_3^TM_1+M_1M_3& 0& M_3M_1-M_1M_3^T\\
  0& M_4M_3^T& & & -M_3^TM_1+M_1M_3& 0\\
\end{smallmatrix}\right].
\end{aligned}
\end{equation*}
\end{tiny}
It is easy to certify that $M_2M_1+MM_3^T=0$, $M_2^TM+M_1M_3=0$, $M_3M_1-M_1M_3^T=0$,\\
$-M_3^TM_1+M_1M_3=0$,$M_2^TM_3=0$, $-M_3^2=0$, $M_3M_4=0$, $M_4M_3^T=0$,
so $A_1A_2+A_2A_3=0$.

  Since $A_1$ and $A_3$ are skew symmetric,
$$-A_2^TA_1-A_3A_2^T=A_2^TA_1^T+A_3^TA_2^T=(A_1A_2+A_2A_3)^T=0.$$

\begin{equation*}
A_1^2-A_2A_2^T=
\begin{tiny}
\left[\begin{smallmatrix}
 X_1& 0& X_5\\
 0& X_2& 0& X_6\\
 {X_5}^T& 0& X_3& 0& X_6 \\
 &  {X_6}^T& 0& X_3& 0& X_6\\
 & &  {X_6}^T& 0& X_3& 0& X_6\\
 & & & \ddots& \ddots& \ddots& \ddots& \ddots&\\
 & & & & & & 0& X_6\\
 & & & & & 0& X_3& 0\\
 & & & & & {X_6}^T& 0& X_4
 \end{smallmatrix}\right]
-
 \left[\begin{smallmatrix}
 M^2+M_3M_3^T& & & & MM_4^T-M_3M_1\\
 & M_1^2\\
 & & M_1^2\\
 & & & \ddots\\
 M_4M-M_1M_3^T& & &  & M_4M_4^T+M_1^2\\
\end{smallmatrix}\right],
\end{tiny}
\end{equation*}
where $X_1=-M_2M_2^T$,
$X_2=-M_2^TM_2-M_3M_3^T$,
$X_3=-M_3^TM_3-M_3M_3^T$,
$X_4=-M_3^TM_3$,
$X_5=M_2M_3$,
$X_6={M_3}^2$.

Note that, $X_5=M_2M_3=0$, $X_6={M_3}^2=0$, $X_2=-M_2^TM_2-M_3M_3^T=-4M$,
$X_3=-M_3^TM_3-M_3M_3^T=-4M$,$MM_4^T-M_3M_1=0$, $M_4M-M_1M_3^T=0$, $M_1^2=M$.

\begin{tiny}
\begin{equation*}
\begin{aligned}
&A_1^2-A_2A_2^T=\\
&\left[\begin{smallmatrix}
X_1- M^2-M_3M_3^T\\
 & X_2-M_1^2\\
 & & X_3-M_1^2\\
 & & & \ddots\\
 & & & & X_3-M_1^2\\
 & & & & & X_4-M_4M_4^T-M_1^2
\end{smallmatrix}\right].
\end{aligned}
\end{equation*}
\end{tiny}

  Here, $X_1- M^2-M_3M_3^T=-5M$, $X_4-M_4M_4^T-M_1^2=-5M$,
$X_2-M_1^2=X_3-M_1^2=-5M$, thus $A_1^2-A_2A_2^T=-5I_{\frac{n}{2}}$.

  Similarly, we can prove $A_3^2-A_2^TA_2=-5I_{\frac{n}{2}}$,
so $S(G_{26}^\sigma)^TS(G_{26}^\sigma)=-S(G_{26}^\sigma)^2=5I_n$. From Lemma \ref{le:c-1},
we obtain that $G_{26}^\sigma$ has the optimum orientation.

Each graph of class (\uppercase\expandafter{\romannumeral 14}) has an optimum orientation.

 Since $G_{31}\cong G_{32}$, Let the rows of $S(G_{31}^\sigma)$
correspond successively the vertices
$v_1,v_2,$ $\ldots,v_{24}$.

\begin{equation*}
S(G_{31}^\sigma)=
\begin{footnotesize}
\left[\begin{smallmatrix}
0& 1& 1& 1& 1& 1& 0& 0& 0& 0& 0& 0& 0& 0& 0& 0& 0& 0& 0& 0& 0& 0& 0& 0\\
-1& 0& 0& 0& 0& 0& -1& 1& 1& -1& 0& 0& 0& 0& 0& 0& 0& 0& 0& 0& 0& 0& 0& 0\\
-1& 0& 0& 0& 0& 0& 1& 0& 0& 0& -1& -1& 1& 0& 0& 0& 0& 0& 0& 0& 0& 0& 0& 0\\
-1& 0& 0& 0& 0& 0& 0& -1& 0& 0& 1& 0& 0& 1& -1& 0& 0& 0& 0& 0& 0& 0& 0& 0\\
-1& 0& 0& 0& 0& 0& 0& 0& -1& 0& 0& 1& 0& -1& 0& 1& 0& 0& 0& 0& 0& 0& 0& 0\\
-1& 0& 0& 0& 0& 0& 0& 0& 0& 1& 0& 0& -1& 0& 1& -1& 0& 0& 0& 0& 0& 0& 0& 0\\
0& 1& -1& 0& 0& 0& 0& 0& 0& 0& 0& 0& 0& 0& 0& 0& 1& -1& -1& 0& 0& 0& 0& 0\\
0& -1& 0& 1& 0& 0& 0& 0& 0& 0& 0& 0& 0& 0& 0& 0& 1& 0& 0& -1& -1& 0& 0& 0\\
0& -1& 0& 0& 1& 0& 0& 0& 0& 0& 0& 0& 0& 0& 0& 0& 0& -1& 0& 1& 0& -1& 0& 0\\
0& 1& 0& 0& 0& -1& 0& 0& 0& 0& 0& 0& 0& 0& 0& 0& 0& 0& 1& 0& -1& -1& 0& 0\\
0& 0& 1& -1& 0& 0& 0& 0& 0& 0& 0& 0& 0& 0& 0& 0& 0& -1& 0& -1& 0& 0& 1& 0\\
0& 0& 1& 0& -1& 0& 0& 0& 0& 0& 0& 0& 0& 0& 0& 0& 0& 0& -1& 1& -1& 0& 0& 0\\
0& 0& -1& 0& 0& 1& 0& 0& 0& 0& 0& 0& 0& 0& 0& 0& -1& 0& 0& 0& -1& 0& 1& 0\\
0& 0& 0& -1& 1& 0& 0& 0& 0& 0& 0& 0& 0& 0& 0& 0& 0& 0& 0& 0& -1& 1& -1& 0\\
0& 0& 0& 1& 0& -1& 0& 0& 0& 0& 0& 0& 0& 0& 0& 0& -1& -1& 0& 0& 0& 1& 0& 0\\
0& 0& 0& 0& -1& 1& 0& 0& 0& 0& 0& 0& 0& 0& 0& 0& 0& -1& 1& 0& 0& 0& -1& 0\\
0& 0& 0& 0& 0& 0& -1& -1& 0& 0& 0& 0& 1& 0& 1& 0& 0& 0& 0& 0& 0& 0& 0& -1\\
0& 0& 0& 0& 0& 0& 1& 0& 1& 0& 1& 0& 0& 0& 1& 1& 0& 0& 0& 0& 0& 0& 0& 0\\
0& 0& 0& 0& 0& 0& 1& 0& 0& -1& 0& 1& 0& 0& 0& -1& 0& 0& 0& 0& 0& 0& 0& -1\\
0& 0& 0& 0& 0& 0& 0& 1& -1& 0& 1& -1& 0& 0& 0& 0& 0& 0& 0& 0& 0& 0& 0& -1\\
0& 0& 0& 0& 0& 0& 0& 1& 0& 1& 0& 1& 1& 1& 0& 0& 0& 0& 0& 0& 0& 0& 0& 0\\
0& 0& 0& 0& 0& 0& 0& 0& 1& 1& 0& 0& 0& -1& -1& 0& 0& 0& 0& 0& 0& 0& 0& -1\\
0& 0& 0& 0& 0& 0& 0& 0& 0& 0& -1& 0& -1& 1& 0& 1& 0& 0& 0& 0& 0& 0& 0& -1\\
0& 0& 0& 0& 0& 0& 0& 0& 0& 0& 0& 0& 0& 0& 0& 0& 1& 0& 1& 1& 0& 1& 1& 0
\end{smallmatrix}\right].
\end{footnotesize}
\end{equation*}

    It is easy to certify that, $S(G_{31}^\sigma)^TS(G_{31}^\sigma)=5I_{24}$.


The graphs of classes (\uppercase\expandafter{\romannumeral 1}) and
(\uppercase\expandafter{\romannumeral 2}) have no optimum orientations.

 From Lemma 3.1, we know that for $G_1\cong K_6$, if $S(G_1^\sigma)^TS(G_1^\sigma)=5I_n$,
then $n=0 (mod \ 4)$. While $G_1$ has 6 vertices, so the graph $G_1$ has no optimum orientations.

  We will prove that the class of graphs $G_3$ with order $n\geq 12$ have no optimum
orientations. We prove by contradiction.

  Otherwise,
let \{$(v_2, v_1),(v_2, v_3),(v_2, v_4),(v_2, v_5),(v_2,
v_6)\}\subset E(G_3^\sigma)$. (We give a reversal of $G_3^\sigma$ at $v_1, v_3,
v_4, v_5, v_6$ if necessary.)

  Let $(v_4, v_7)\in E(G_3^\sigma)$. (We give a reversal of $G_3^\sigma$ at $v_7$ if necessary.)
By applying Lemma \ref{le:o-3}, we have $w_{v_2v_7}^{+}(2)=w_{v_2v_7}^{-}(2)$, thus
$(v_7, v_5)\in E(G_3^\sigma)$.

  Let $(v_4, v_8)\in E(G_3^\sigma)$. (We give a reversal of $G_3^\sigma$ at $v_8$ if necessary.)
By applying Lemma \ref{le:o-3}, we have $w_{v_2v_8}^{+}(2)=w_{v_2v_8}^{-}(2)$, thus
$(v_8, v_6)\in E(G_3^\sigma)$.

  Let $(v_6, v_9)\in E(G_3^\sigma)$. (We give a reversal of $G_3^\sigma$ at $v_9$ if necessary.)
By applying Lemma \ref{le:o-3}, we have $w_{v_2v_9}^{+}(2)=w_{v_2v_9}^{-}(2)$, thus
$(v_9, v_5)\in E(G_3^\sigma)$.

  Now we are not sure whether $(v_1, v_4)\in E(G_3^\sigma)$,
so we discuss the following two cases.

 \textit{Case 1.} $(v_1, v_4)\in E(G_3^\sigma)$.

By Lemma \ref{le:o-3}, we have $w_{v_1v_7}^{+}(2)=w_{v_1v_7}^{-}(2)$ and
$w_{v_1v_9}^{+}(2)=w_{v_1v_9}^{-}(2)$, then $(v_1, v_5)$ $\in E(G)$ and
$(v_1, v_6)\in E(G_3^\sigma)$, respectively. Now we find the four paths
$v_1v_3v_2$, $v_1v_4v_2$, $v_1v_5v_2$, $v_1v_6v_2$ of length 2 between
$v_1$ and $v_2$. Since $w_{v_1v_2}^{-}(2)=3$, $w_{v_1v_2}^{+}(2)\neq
w_{v_1v_2}^{-}(2)$. Then this contradicts to Lemma \ref{le:o-3}.

 \textit{Case 2.} $(v_1, v_4)\notin E(G_3^\sigma)$.

Then $(v_4, v_1)\in E(G_3^\sigma)$. By applying Lemma \ref{le:o-3}, we have
$w_{v_1v_7}^{+}(2)=w_{v_1v_7}^{-}(2)$, thus $(v_5, v_1)\in E(G_3^\sigma)$.
Similarly $w_{v_1v_9}^{+}(2)=w_{v_1v_9}^{-}(2)$, thus $(v_6, v_1)\in E(G_3^\sigma)$.
Now we find the four paths $v_1v_3v_2$,\\ $v_1v_4v_2$,$v_1v_5v_2,
v_1v_6v_2$ of length 2 between $v_1$ and $v_2$.
Since $w_{v_1v_2}^{+}(2)=3$, $w_{v_1v_2}^{+}(2)\neq w_{v_1v_2}^{-}(2)$.
Thus this contradicts to Lemma \ref{le:o-3}.

  From Cases 1 and 2, we know whether $(v_1, v_4)\in E(G_3^\sigma)$ or not,
there is a contradiction to Lemma \ref{le:o-3}. Thus the graphs (\uppercase\expandafter{\romannumeral 1}) and
(\uppercase\expandafter{\romannumeral 2}) have no optimum orientations.

The graphs of class (\uppercase\expandafter{\romannumeral 11}) have no optimum orientations.

  For $G_{27}$, we will prove that there is not an optimum orientation by contradiction.

    Otherwise let \{$(v_4, v_1),(v_4, v_8),(v_4, v_{11}),(v_4, v_{14}),(v_4, v_{15})\}\subset
E(G_{27}^\sigma)$ (we give a reversal of $G_{27}^\sigma$ at $v_1, v_8, v_{11}, v_{14}, v_{15}$
if necessary).

   Let $(v_8, v_{16})\in E(G_{27}^\sigma)$ (we give a reversal of $G_{27}^\sigma$ at
$v_{16}$ if necessary). By applying Lemma \ref{le:o-3},
$w_{v_4v_{16}}^{+}(2)=w_{v_4v_{16}}^{-}(2)$, thus $(v_{16}, v_{11})\in
E(G_{27}^\sigma)$. Let $(v_{16}, v_7)\in E(G_{27}^\sigma),$
$(v_{11},v_3)\in E(G_{27}^\sigma)$ (we give a reversal of
$G_{27}^\sigma$ at $v_7$ or $v_3$ if necessary). By applying Lemma \ref{le:o-3},
$w_{v_7v_{11}}^{+}(2)=w_{v_7v_{11}}^{-}(2)$, thus $(v_3, v_7)\in
E(G_{27}^\sigma)$, $w_{v_1v_{11}}^{+}(2)=w_{v_1v_{11}}^{-}(2)$,
thus $(v_3, v_1)\in E(G_{27}^\sigma)$.

   Let $(v_7, v_2)\in E(G_{27}^\sigma)$ (we give a reversal of $G_{27}^\sigma$ at $v_2$
if necessary). By applying Lemma \ref{le:o-3},
$w_{v_1v_{7}}^{+}(2)=w_{v_1v_7}^{-}(2)$, thus $(v_2, v_1)\in E(G_{27}^\sigma)$,
$w_{v_7v_8}^{+}(2)=w_{v_7v_8}^{-}(2)$, thus $(v_8, v_2)\in E(G_{27}^\sigma)$.

   Let $(v_1, v_6)\in E(G_{27}^\sigma)$, $(v_2, v_{10})\in E(G_{27}^\sigma)$
(we give a reversal of $G_{27}^\sigma$ at $v_6$ or $v_{10}$ if necessary). By applying Lemma \ref{le:o-3},
$w_{v_2, v_6}^{+}(2)=w_{v_2, v_6}^{-}(2)$, thus $(v_6, v_{10})\in E(G_{27}^\sigma)$,
$w_{v_4, v_6}^{+}(2)=w_{v_4, v_6}^{-}(2)$, thus $(v_6, v_{15})\in E(G_{27}^\sigma)$,

   Let $(v_6, v_{13})\in E(G_{27}^\sigma)$ (we give a reversal of $G_{27}^\sigma$
at $v_{13}$ if necessary).
By applying Lemma \ref{le:o-3}, $w_{v_3, v_6}^{+}(2)=w_{v_3, v_6}^{-}(2)$,
thus $(v_3, v_{13})\in E(G_{27}^\sigma)$.

   Let $(v_8, v_{12})\in E(G_{27}^\sigma)$ (we give a reversal of $G_{27}^\sigma$
at $v_{12}$ if necessary).
By applying lemma 3.3, $w_{v_4, v_{12}}^{+}(2)=w_{v_4, v_{12}}^{-}(2)$,
thus $(v_{12}, v_{15})\in E(G_{27}^\sigma)$, $w_{v_6, v_{12}}^{+}(2)=w_{v_6, v_{12}}^{-}(2)$,
thus $(v_{10}, v_{12})\in E(G_{27}^\sigma)$.

   Let $(v_{15}, v_9)\in E(G_{27}^\sigma)$ (we give a reversal of $G_{27}^\sigma$
at $v_9$ if necessary).
By applying Lemma \ref{le:o-3}, $w_{v_6, v_9}^{+}(2)=w_{v_6, v_9}^{-}(2)$,
thus $(v_9, v_{13})\in E(G_{27}^\sigma)$. $w_{v_4, v_9}^{+}(2)=w_{v_4, v_9}^{-}(2)$,
thus $(v_9, v_{11})\in E(G_{27}^\sigma)$.

   Let $(v_5, v_{14})\in E(G_{27}^\sigma)$ (we give a reversal of $G_{27}^\sigma$
at $v_5$ if necessary).
By applying Lemma \ref{le:o-3}, $w_{v_1, v_{14}}^{+}(2)=w_{v_1, v_{14}}^{-}(2)$,
thus $(v_1, v_5)\in E(G_{27}^\sigma)$.

   Now if $(v_{13}, v_{14})\in E(G_{27}^\sigma)$, Applying Lemma \ref{le:o-3} to
$w_{v_9, v_{14}}^{+}(2)=w_{v_9, v_{14}}^{-}(2)$, we have $(v_5, v_9)\in E(G_{27}^\sigma)$.
Similarly $w_{v_5, v_{15}}^{+}(2)=w_{v_5, v_{15}}^{-}(2)$
implies $(v_5, v_{12})\in E(G_{27}^\sigma)$,
$(v_{14}, v_7)$ $\in E(G_{27}^\sigma)$ ($w_{v_3, v_{14}}^{+}$
$(2)=w_{v_3, v_{14}}^{-}(2)$).
$(v_{16}, v_5)\in E(G_{27}^\sigma)$ ($w_{v_5, v_{11}}^{+}(2)=w_{v_5, v_{11}}^{-}(2)$).
$(v_2, $ $v_9)\in E(G_{27}^\sigma)$ ($w_{v_1, v_9}^{+}(2)=w_{v_1, v_9}^{-}(2)$).
$(v_{15}, v_7)\in E(G_{27}^\sigma)$ ($w_{v_7, v_9}^{+}(2)=w_{v_7, v_9}^{-}(2)$).
$(v_3, v_{12})\in E(G_{27}^\sigma)$ ($w_{v_3, v_{15}}^{+}(2)=w_{v_3, v_{15}}^{-}(2)$).
$(v_8, v_{13})\in E(G_{27}^\sigma)$ ($w_{v_4, v_{13}}^{+}(2)=w_{v_4, v_{13}}^{-}$ $(2)$).
$(v_{16}, v_6)$ $\in E(G_{27}^\sigma)$ ($w_{v_6, v_8}^{+}(2)=w_{v_6, v_8}^{-}(2)$).
$(v_{10}, v_{11})\in E(G_{27}^\sigma)$ ($w_{v_6, v_{11}}^{+}(2)=w_{v_6, v_{11}}^{-}(2)$).

   Now we obtain the oriented graph, while $v_9v_{11}v_{10}$, $v_9v_2v_{10}$ are the two negative
paths adjoining $v_9$ and $v_{10}$ and $|N(v_9)\cap N(v_{10})|=2$,
this contradicts to Lemma \ref{le:o-3}.

   If $(v_{14}, v_{13})\in E(G_{27}^\sigma)$, By applying Lemma \ref{le:o-3},
$(v_9, v_5)\in E(G_{27}^\sigma)$ ($w_{v_9, v_{14}}^{+}(2)=w_{v_9, v_{14}}^{-}$ $(2)$),
$(v_{12}, v_5)\in E(G_{27}^\sigma)$ ($w_{v_5, v_{15}}^{+}(2)=w_{v_5, v_{15}}^{-}(2)$),
$(v_7, v_{14})\in E(G_{27}^\sigma)$ ($w_{v_3, v_{14}}^{+}(2)$ $=w_{v_3, v_{14}}^{-}(2)$),
$(v_5, v_{16})\in E(G_{27}^\sigma)$ ($w_{v_5, v_{11}}^{+}(2)=w_{v_5, v_{11}}^{-}(2)$),
$(v_9, v_2)\in E(G_{27}^\sigma)$ ($w_{v_1, v_9}^{+}(2)$ $=w_{v_1, v_9}^{-}(2)$).
$(v_7, v_{15})\in E(G_{27}^\sigma)$ ($w_{v_7, v_9}^{+}(2)=w_{v_7, v_9}^{-}(2)$).
$(v_{12}, v_3)\in E(G_{27}^\sigma)$ ($w_{v_3, v_{15}}^{+}(2)$ $=w_{v_3, v_{15}}^{-}(2)$).
$(v_{13}, v_8)\in E(G_{27}^\sigma)$ ($w_{v_4, v_{13}}^{+}(2)=w_{v_4, v_{13}}^{-}(2)$).
$(v_6, v_{16})\in E(G_{27}^\sigma)$ ($w_{v_6, v_8}^{+}$ $(2)$ $=w_{v_6, v_8}^{-}(2)$).
$(v_{11}, v_{10})\in E(G_{27}^\sigma)$ ($w_{v_6, v_{11}}^{+}(2)=w_{v_6, v_{11}}^{-}(2)$).

   Now we obtain the oriented graph, while $v_9v_{11}v_{10}$, $v_9v_2v_{10}$ are the two positive
paths adjoining $v_9$ and $v_{10}$ and $|N(v_9)\cap N(v_{10})|=2$,
this contradicts to Lemma \ref{le:o-3}.

   Hence the graph $G_{27}$ has no optimum orientations.
\end{proof}

\end{document}